\documentclass[11pt]{amsart}
\pdfoutput=1

  \pdfpageattr{/Group <</S /Transparency /I true /CS /DeviceRGB>>}

\usepackage[T1]{fontenc}
\usepackage[english]{babel}
\usepackage{geometry}
\usepackage{amsfonts,amsmath,amssymb,amsthm,mathrsfs}
\usepackage{caption}
\usepackage[labelfont=rm]{subcaption}
\usepackage{xspace}
\usepackage{microtype}
\usepackage{enumerate}
\usepackage{booktabs}
\usepackage{hyperref}
\usepackage{algorithm}
\usepackage{algpseudocode}
\usepackage{paralist}

\usepackage{tikz}
\usetikzlibrary{arrows}
\usetikzlibrary{arrows.meta}
\usepackage{pgfplots}

\microtypesetup{tracking,kerning,spacing}
\microtypecontext{spacing=nonfrench}

\hypersetup{colorlinks,linkcolor=blue ,citecolor=blue, urlcolor=blue}

\newcommand\polymake{\texttt{polymake}\xspace}
\newcommand\topcom{\texttt{TOPCOM}\xspace}
\newcommand\mptopcom{\texttt{MPTOPCOM}\xspace}
\newcommand\mptopcomone{\texttt{MPTOPCOM-1}\xspace}
\newcommand\gfan{\texttt{Gfan}\xspace}
\newcommand\mts{\texttt{mts}\xspace}

\newcommand\mpi{\texttt{MPI}\xspace}
\newcommand\cddlib{\texttt{cddlib}\xspace}
\newcommand\valgrind{\texttt{Valgrind}\xspace}
\newcommand\massif{\texttt{massif}\xspace}
\newcommand\callgrind{\texttt{callgrind}\xspace}
\newcommand\permlib{\texttt{permlib}\xspace}

\DeclareSymbolFontAlphabet{\mathbb}{AMSb}

\newcommand\RR{{\mathbb R}}
\newcommand\TT{{\mathbb T}}
\newcommand\1{{\mathbf 1}}

\newcommand\SetOf[2]{\left\{\left.#1\vphantom{#2}\ \right|\ #2\vphantom{#1}\right\}}
\newcommand\smallSetOf[2]{\{{#1}\,|\,{#2}\}}
\DeclareMathOperator{\vol}{vol}
\DeclareMathOperator{\Sym}{Sym}

\newcommand{\id}{{\operatorname{id}}}
\newcommand{\mydepth}{\mathit{depth}}
\newcommand{\mynull}{\mathit{null}}

\newcommand{\eval}{v}               
\newcommand{\graph}{\Gamma}         
\newcommand{\group}{G}              
\newcommand{\groupElem}{g}          
\newcommand{\jbound}{\psi}          
\newcommand{\switchTableSize}{\mu}  

\newcommand{\secondaryPoly}{\Sigma{\operatorname{-poly}}}
\newcommand{\secondaryFan}{\Sigma{\operatorname{-fan}}}
\newcommand{\PhiReg}{\Phi_{\operatorname{reg}}}
\newcommand{\PhiRegComp}{\Phi_{\operatorname{reg}}^c}

\theoremstyle{plain}
    \newtheorem{theorem}{Theorem}
    \newtheorem{corollary}[theorem]{Corollary}
    \newtheorem{lemma}[theorem]{Lemma}
    \newtheorem{proposition}[theorem]{Proposition}
\theoremstyle{definition}
    \newtheorem{remark}[theorem]{Remark}
    \newtheorem{example}[theorem]{Example}
    \newtheorem{definition}[theorem]{Definition}

\title{Parallel Enumeration of Triangulations}

\author{Charles Jordan}

\author{Michael Joswig}

\author{Lars Kastner}

\address[Charles Jordan]{
Graduate School of Information Science and Technology,
 Hokkaido University,
 N-14 W-9, 060-0814 Sapporo, Japan
}
\email{skip@ist.hokudai.ac.jp}
\thanks{%
Research by C.\ Jordan supported in part by JSPS Kakenhi Grants 16H02785 and 18K18027.
}

\address[Michael Joswig]{
Institut f{\"u}r Mathematik,
 TU Berlin,
 Str.\ des 17. Juni 136, 10623 Berlin, Germany
}
\email{joswig@math.tu-berlin.de}

\thanks{%
Research by M. Joswig is carried out in the framework of Matheon supported by Einstein Foundation Berlin. 
Further partial support by Deutsche Forschungsgemeinschaft (SFB-TRR 109: ``Discretization in Geometry and Dynamics'' and SFB-TRR 195: ``Symbolic Tools in Mathematics and their Application'') is gratefully acknowledged.
}

\address[Lars Kastner]{
Institut f{\"u}r Mathematik,
 TU Berlin,
 Str.\ des 17. Juni 136, 10623 Berlin, Germany
}
\email{lkastner@math.tu-berlin.de}

\thanks{%
Research by L. Kastner is supported by Deutsche Forschungsgemeinschaft (SFB-TRR 195: ``Symbolic Tools in Mathematics and their Application'').
}


\DeclareMathOperator{\gkz}{gkz}
\newcommand{\canonical}{\rho}
\DeclareMathOperator{\switch}{st}
\DeclareMathOperator{\Adj}{Adj}
\DeclareMathOperator{\conv}{conv}

\begin{document}

\begin{abstract}
  We report on the implementation of an algorithm for computing the set of all regular triangulations of finitely many points in Euclidean space.
  This algorithm, which we call \emph{down-flip reverse search},
  can be restricted, e.g., to computing full triangulations only; this case is particularly relevant for tropical geometry.
  Most importantly, down-flip reverse search allows for massive parallelization, i.e., it scales well even for many cores.
  Our implementation allows to compute the triangulations of much larger point sets than before.
\end{abstract}

\maketitle

\section{Introduction}
\noindent
Triangulations are ubiquitous in combinatorics, optimization, algebra and other parts of mathematics.
For an overview about the range of applications we recommend the first chapter of the monograph \cite{Triangulations} by De Loera, Rambau and Santos, and the rest of that book is recommended for the foundations.
The software which defines the state of the art is Rambau's \topcom \cite{TOPCOM}; see also \cite{PfeifleRambau03}.
An implementation of a different method is part of \gfan by Jensen \cite{gfan}.
While this comes with some advantages of its own, \topcom is superior in most cases.
In this article we describe an algorithm using different ideas but which still shares some of \topcom's features.
Most importantly we report on new computational results obtained with our implementation \mptopcom; some of these are out of the reach of the other software systems.

In order to explain the method it is best to start by looking into \topcom's algorithm; see also \cite[\S8.3]{Triangulations}.
As a first step, for a fixed point configuration, $P$, \topcom creates one triangulation as a \emph{seed}.
Then \topcom generates all triangulations of $P$ which can be obtained by local transformations, known as \emph{flips}.
The flips induce a graph structure on the set of triangulations of $P$, and \topcom makes a breadth first search starting at the seed.
In practical applications it is often important to restrict the attention to those triangulations of $P$ which are \emph{regular}, i.e., they are induced by a convex lifting function.
Regularity can be checked via solving a linear programming feasibility problem, and this is supported by \topcom, e.g., via calling \cddlib \cite{cddlib}.
Any two regular triangulations of $P$ are connected by a sequence of flips.
Checking for regularity on the way is costly but is also beneficial since the regular triangulations of $P$ have a particularly nice structure.
The reason is that they correspond to the vertices of a convex polytope, the \emph{secondary polytope} of $P$, which was introduced by Gel$'$fand, Kapranov and Zelevinsky \cite{GKZ}, and the edges of the secondary polytope correspond to flips (but there are flips which do not arise from edges).
The coordinates of a regular triangulation $\Delta$, seen as a vertex of the secondary polytope, form the \emph{GKZ-vector} of~$\Delta$.
Already in 2002 Imai et al.~\cite{Imai:2002} proposed an algorithm which exploits the GKZ-vectors to tailor a reverse search scheme \cite{AF93} for enumerating the regular triangulations of $P$, and we call their method \emph{down-flip reverse search}.
Surprisingly, the down-flip reverse search algorithm apparently received little attention, and some of its properties have been discovered independently by Pournin and Liebling \cite{PourninLiebling:2007}.
A key challenge in practice is that relevant point sets $P$ exhibit a great deal of symmetry.
Usually this symmetry is given in terms of generators of some finite group $\group$ acting on $P$ by permutations.
The real task now is to compute the (regular) triangulations of $P$ up to symmetry, i.e., exactly one representative from each $\group$-orbit.
This reduction is strictly necessary since the total number of triangulations is often too large.
Imai et al.~\cite{Imai:2002} also describe down-flip reverse search up to symmetry.

One of the main advantages of the reverse search algorithm is that it is output
sensitive. In particular, if one has bounds on the number of neighbors of a
node in the search graph, one can derive an effective upper bound on the memory
consumption of reverse search, see \cite[Theorem 3.2]{AF93}. The price one pays
is recomputation of intermediate results, whenever a node is analyzed, but not
visited. This obstacle can be partially overcome by using caches. While
\topcom's limit often is dictated by the amount of memory available,
\mptopcom's limit is given by the time one is willing to let the program run.
This is where parallelization comes in.

While the general algorithmic idea is known, a practical implementation of
down-flip reverse search requires one to overcome several challenges.  That
is the focus of this paper.
We will  show that our implementation \mptopcom of down-flip reverse search, which can enumerate hundreds of millions of triangulations of a given point set, is superior to other methods in practice.
Two aspects are most important.
First, for large input, the implementation must be parallelized in order to be able to benefit from modern hardware.
This is where the reverse search scheme shows its full strength; \mts \cite{mts_tutorial} is a competitive parallel implementation of the abstract reverse search method based on \mpi \cite{mpi}.
Second, by far the most frequent subtask is to compute a canonical representative for each $\group$-orbit of a given triangulation.
This allows to distinguish two orbits by comparing their representatives.
A minor drawback of parallelized reverse search is that those representatives are recomputed by several workers, even for the same orbit and the same triangulation.
This is because we want to avoid \emph{any} communication between the nodes in the computation tree since this is what makes parallelized reverse search so successful.
Our main new algorithmic contribution is a procedure for computing these canonical representatives in a way that is ---in our particular situation--- much faster than standard methods from computational group theory.
Despite our highly optimized setup, depending on the input and the desired output, the computation of the canonical representatives still may take more than 90\% of the combined total time.
The benefit of our approach, however, is tremendous.
It turns out that, even if we restrict our parallel implementation to a single core, we are able to beat \topcom by a factor of five or more on medium size input.
More importantly, on large input \mptopcom scales rather well even for more than a hundred cores.
This means that we are now able to compute triangulations of point sets an order of magnitude larger than before.

This paper is organized as follows.
In Section~\ref{sec:reverse} we collect the basic facts concerning the reverse search method of Avis and Fukuda \cite{AF93}.
This is a powerful general scheme for organizing a large enumeration via a rooted tree.
The paradigmatical example is a dual convex hull computation.
Here the input is a system of linear inequalities, and the output are the vertices and the rays of the feasible region, which is a convex polyhedron.
A generic linear objective function induces an orientation on the edges of that polyhedron, and any pivoting strategy for the simplex method yields a tree, in which each edge is directed toward the global optimum.
The basic idea of down-flip reverse search is to mimic this behavior on the secondary polytope of a point configuration.
We close this section by describing the principles underlying the parallelization employed by \mts \cite{mts_tutorial}.
Section~\ref{sec:triangulations} deals with the basic notions concerning triangulations, and we formulate the down-flip reverse search algorithm.
Then, in Section~\ref{sec:symmetry} we briefly explain the general approach of Imai et al.\ \cite{Imai:2002} to combine enumeration up to symmetry with down-flip reverse search.
The core of our paper is Section~\ref{sec:canonical}.
\emph{The} major challenge in \mptopcom's implementation is to efficiently find canonical representatives for each orbit of triangulations.
To this end we introduce the concept of \emph{switch tables} and \emph{evaluations} to facilitate that computation (cf.\ Definition~\ref{def:switch}).
This is developed in a way which specifically addresses the enumeration of triangulations, and yet we believe that it may also be useful in other circumstances.
Our main theoretical contribution is a procedure for computing canonical representatives via switch tables (cf.\ Algorithm~\ref{alg:canonical-rep}) and its analysis.
In Section~\ref{sec:experiments} we report on experimental results.
First we consider a few standard examples, such as the 16 vertices of the 4-dimensional regular cube.
More interesting are our new results.
This concerns, e.g., products of simplices; these computations have been employed by Schr\"oter for deriving results on coarsest matroid subdivision of hypersimplices \cite{Schroeter:1707.02814}.
Further, for the first time we are able to enumerate all regular and full triangulations of the 3-simplex with dilation factor of three.
This is a configuration of 20 points in $\RR^3$, and it turns out that it has precisely $21\, 125\, 102$ regular and full triangulations, up to the symmetry induced by the symmetric group of degree four, acting on the vertices of the tetrahedron (cf.\ Theorem~\ref{thm:cubic-surfaces}).
This outcome is relevant, e.g., for tropical geometry, as this leads to the classification of the smooth tropical cubic surfaces in the tropical 3-torus; see \cite[\S4.5]{Tropical+Book}.
A full account of the consequences from that one computation is beyond the scope of this paper.

\mptopcom is available as open source software which can be downloaded from \url{https://www.polymake.org/mptopcom}.

We are grateful to J\"org Rambau, Francisco Santos, Benjamin Schr\"oter and an anonymous reviewer for valuable comments and suggestions and to Benjamin Lorenz for assistance with the implementation.

\section{Budgeted reverse search and its parallelization}
\label{sec:reverse}
\noindent
Reverse search~\cite{AF93} is a technique for enumerating large sets of
objects.  Essentially, one explores a graph $\graph=(V,E)$ where $V$ is the set of
objects to be enumerated and the edges are given by an adjacency oracle.  For
every node $v$, the number $\delta(v)$ denotes the outdegree of $v$ in $\graph$. The adjacency
oracle $\Adj(v,j)$ returns the $j$th neighbor of $v\in V$ for $j\in[\delta]$
or $\mynull$ if no such neighbor exists.  To apply reverse search to a problem,
one provides the adjacency oracle, a local search function $\pi(v)$ and an
element $v^*\in V$.  It is required that $\pi(v)$ returns a tuple $(u,j)$
such that $\Adj(u,j)=v$ and that repeated application of $\pi$ to any $v\in V$
results in a path from $v$ to $v^*$.  The local search function therefore
generates a spanning tree of $\graph$ with root $v^*$.  We sometimes omit
the second component of $\pi(v)$ when it is not needed, e.g. in line 6 of
Algorithm~\ref{alg:reverse}.

\begin{algorithm}
\begin{algorithmic}[1]
\Procedure{RS}{$v^*$, $\Adj$, $\pi$}
        \State $v \gets v^*,\  j \gets 0,\  \mydepth \gets 0$
        \Repeat
        \While {$j < \delta(v)$} \label{Step:neighbor_list}
                \State $j \gets j+1$
                \If {$\pi(\Adj(v,j)) = v$} \label{Step:duplicate_work}
                        \State $v \gets \Adj(v,j)$
                        \State $j \gets 0$
                        \State $\mydepth \gets \mydepth+1$
                        \State {Output $v$}
                \EndIf
        \EndWhile
        \If {$\mydepth > 0$}
                \State $(v,j) \gets \pi(v)$
                \State $\mydepth \gets \mydepth-1$
        \EndIf
        \Until {$\mydepth=0$ and $j=\delta(v)$}
\EndProcedure
\end{algorithmic}
\caption{Reverse Search}
\label{alg:reverse}
\end{algorithm}

Reverse search (cf. Algorithm~\ref{alg:reverse}) traverses this spanning tree in depth-first fashion starting from $v^*$, and at each node outputs the object corresponding to that node.
One can use the local search function to backtrack.
This allows to implement reverse search in a way which does not require any additional memory after an initial setup.
However, it is also possible to cache auxiliary information needed by the application, for example to avoid recomputation when backtracking.
This way the memory usage can be controlled even for computations with extremely large output.
For more details on reverse search see~\cite{AF93}.

It was recognized from the beginning that reverse search can be easily
parallelized.  The enumeration process can be restarted from any node given
only a description of that node, and no node is reachable via multiple paths.
This means that one does not have to store previously-visited nodes, and
communication between processes is minimal.  Many reverse search trees are
highly unbalanced and so the underlying problem is to explore an unbalanced
tree in parallel.

The first parallel reverse search implementation was the generic reverse search
layer in ZRAM~\cite{BMFN99} which was applied in various areas,
including the first parallel program for vertex/facet enumeration~\cite{BMFN99}
and a program for certain quadratic maximization problems~\cite{FKL05}.
Other parallel reverse search applications include the computation of
Minkowski sums~\cite{Weibel10}.
Furthermore, Jensen computed exact mixed volumes \cite{Jensen:2016} and homotopy continuation \cite{Jensen:1601.02818} via parallel reverse search implemented in \gfan \cite{gfan}.

Recently budgeted reverse search~\cite{mplrs} was introduced as a
simple scheme for load-balancing in parallel reverse search.  There, one
adds an additional parameter (the \emph{budget}) and each of the parallel
processes explores its assigned subtree subject to the budget.  Once the
budget is exhausted, the process backtracks to the root of the subtree
while reporting unexplored nodes along the backtrack path.  These unexplored
nodes are scheduled for later exploration.  This is a particularly simple
approach to parallel tree exploration that in practice can scale beyond 
1000 cores~\cite{mplrs}.  If properties of the trees generated by an
application are known, it can be possible to prove certain performance
guarantees~\cite{AD17}.

Budgeted reverse search inherits a number of other features from reverse
search.  One can checkpoint and restart the overall process by waiting for
all processes to exhaust their budget and writing the descriptions of
unexplored nodes to a file.  The budget can be tuned dynamically based
on the number of unexplored subtrees available.  While the implementation
of~\cite{mplrs} is specific to vertex/facet enumeration, a generic
implementation of parallel budgeted reverse search (\mts) is also
available along with a tutorial~\cite{mts_tutorial}.
 
Implementing a budgeted reverse search application with \mts allows a
clean separation of the parallelization layer and application code.  This
permits independent development of the application and the parallelization 
layer. 
The current \mts implementation uses \mpi~\cite{mpi} and
dedicates a process as \emph{master} and another to handling output.
This overhead is insignificant
when many processes are available, but limits parallel efficiency when only
few processes are used.
We refer to Section~\ref{sec:experiments} below for further details on how \mptopcom employs \mts for enumerating triangulations.

\section{Triangulations of point configurations}
\label{sec:triangulations}
\noindent
Let $P\subset\RR^d$ be a finite set of $n$ points that affinely spans the entire space.
A \emph{(polyhedral) subdivision} $\Sigma$ of $P$ is a polytopal complex which covers the convex hull $\conv P$, such that the vertices of each cell form a subset of the given points $P$; cf.\ \cite[\S2.3.1]{Triangulations}.
The subdivision $\Sigma$ is \emph{regular} if it is induced by a height function $h:P\to\RR$ in the sense that the lower convex hull of
\[
\conv\bigl\{ (p,h(p)) \mid  p\in P \bigr\} \quad \subset \ \RR^{d+1}
\]
projects to $\Sigma$ by omitting the last coordinate.
The subdivision $\Sigma$ is a \emph{triangulation} if all its cells are simplices.
The set of all subdivisions of $P$ is partially ordered by refinement, and the triangulations are precisely the finest subdivisions.
Our goal in this section is to discuss and present algorithms for enumerating all regular triangulations of $P$.

For a given subdivision $\Sigma$ the set of all height functions which induce $\Sigma$ on $P$ is a relatively open polyhedral cone, the \emph{secondary cone} of $\Sigma$.
The secondary cone of $\Sigma$ is non-empty if and only if $\Sigma$ is regular.
The set of all secondary cones forms a polyhedral fan, the \emph{secondary fan} $\secondaryFan(P)$.
The relatively open secondary cones partition the space $\RR^n$ of height functions, i.e., the secondary fan is \emph{complete}.
Each non-empty secondary cone contains a $(d+1)$-dimensional linear subspace, the space of linealities of the secondary fan.
Fixing the heights on an affine basis in $P$ amounts to intersecting the secondary fan of $P$ in a way such that each cone in the resulting fan is pointed.
We call that $(n-d-1)$-dimensional pointed polyhedral fan the \emph{pointed secondary fan} of $P$.
It is unique up to linear transformations.

As a key fact the pointed secondary fan is \emph{polytopal}, i.e., it is the normal fan of a convex polytope of dimension $n-d-1$.
We quickly review the construction.
For a triangulation $\Delta$ of $P$ the \emph{GKZ-vector} is
\[
  \gkz_\Delta \ = \ \bigl( \gkz_\Delta(p) \mid p \in P \bigr) \enspace ,
\]
where $\gkz_\Delta(p)$ is the sum of the (Euclidean) volumes of those simplices in $\Delta$ which contain $p$ as a vertex.
The convex hull
\[
  \secondaryPoly(P) \ = \ \conv\SetOf{\gkz_\Delta}{\Delta \text{ triangulation of } P}
\]
of all GKZ-vectors is the \emph{secondary polytope} of $P$; its normal fan is the (pointed) secondary fan of $P$; cf.\ \cite[\S5.2.2]{Triangulations}.
The vertices of $\secondaryPoly(P)$ are precisely the GKZ-vectors of the regular triangulations.
As we assumed the point configuration $P$ is spanning, the dimension of its secondary polytope equals $n-d-1$; cf. \cite[\S5.1.3]{Triangulations}.

We will now sketch an algorithm by Imai et al.~\cite{Imai:2002} to enumerate (regular) triangulations of $P$ based on applying reverse search to the vertex--edge graph of $\secondaryPoly(P)$ (or a suitable supergraph).
See \cite[\S5.3.2]{Triangulations} for an account of the geometric facts.
We begin by defining a linear objective function on the secondary polytope $\secondaryPoly(P)$.
To this end choose a positive real number $M$.
Then the vector
\[
\lambda \ = \ (M^n,M^{n-1},\dots,M)
\]
of length $n$ defines a linear form on $\RR^n$.
For all sufficiently large $M\gg0$ that linear form is injective on the finite set $\smallSetOf{\gkz_\Delta}{\Delta \text{ triangulation of } P}$.
Moreover, if $M\gg0$, then comparing any two GKZ-vectors with respect to $\lambda$ amounts to checking the lexicographic ordering.
In particular, there is no need to determine any valid choices for $M$.
Like the GKZ-vectors the following relies on the choice of a fixed ordering of the points in $P$.
\begin{definition}\label{def:total}
  For any two triangulations, $\Delta$ and $\Delta'$, of $P$ we let $\Delta'>\Delta$ if
  \[
  \lambda(\gkz_{\Delta'}) \ > \ \lambda(\gkz_\Delta) \enspace ,
  \]
  or if $\lambda(\gkz_{\Delta'}) = \lambda(\gkz_\Delta)$ and $\Delta'$ is lexicographically larger than $\Delta$.
\end{definition}
This defines a total ordering on the set of all triangulations of $P$.
Since $\lambda$ induces a total ordering on the vertices of $\secondaryPoly(P)$ the lexicographic ordering is not required as a tie-breaker if restricted to regular triangulations.
However, it is important for nonregular triangulations; see Example~\ref{exmp:cube} below.

A \emph{flip} $f$ is a local modification of a triangulation, $\Delta$, of $P$ which yields another triangulation, $\Delta'$; we write $f=[\Delta \leadsto \Delta']$.
Each edge of the secondary polytope comes from a flip, but the converse does not hold \cite[\S5.3.1]{Triangulations}.
We call $f$ an \emph{up-flip} if $\Delta'>\Delta$.
Otherwise $\Delta'<\Delta$, and $f$ is a \emph{down-flip}.
The \emph{flip graph} $\Phi$ of $P$ is the graph whose nodes are the triangulations of $P$ and whose edges are given by the flips.
The flip graph $\Phi$ is directed with respect to up-flips.
Note that $\Phi$ is not necessarily connected \cite[\S7.3 and \S7.4]{Triangulations}.
However, the subgraph $\PhiReg$ induced on the subset of regular triangulations is always connected.

\begin{remark}\label{rem:caution_flips}
  It is known that if $f=[\Delta\leadsto\Delta']$ is a flip and $\Delta$ is regular, then $\Delta'$ does not need to be regular.
  That is, the connected component $\PhiRegComp$ of the regular triangulations may be strictly larger than $\PhiReg$.
  We call $\PhiRegComp$ the \emph{regular component} of the flip graph, and we call a triangulation \emph{sub-regular} if it is reachable from a regular triangulation via down-flips.
  Note that sub-regularity depends on the choice of ordering of triangulations.
  Moreover, even if $\Delta$ and $\Delta'$ both are regular then $f$ does not necessarily correspond to an edge of $\secondaryPoly(P)$.
  In general the vertex--edge graph of $\secondaryPoly(P)$ may be a proper subgraph of $\PhiReg$ as it may have fewer edges.
\end{remark}

The reverse search algorithm (cf.\ Algorithm~\ref{alg:reverse}) works on a graph which is given implicitly by means of an oracle.
It is only after the termination of the procedure that the entire graph (or rather a spanning tree) is known.
Therefore, we introduce the following notation.
Our adjacency oracle $\PhiReg(\Delta,j)$ for the regular flip graph $\PhiReg$ at a given regular triangulation $\Delta$ of $P$ returns a regular triangulation which can be obtained from $\Delta$ by some down-flip, and it is the $j$th one in the total ordering from Definition~\ref{def:total}.
If there are less than $j$ regular triangulations accessible via down-flips, then the oracle returns $\mynull$.
We call our predecessor function $\pi$, and it returns the maximal triangulation which can be obtained via an up-flip, if it exists.
There is a unique regular triangulation $\Delta^*$ whose GKZ-vector is maximal among all triangulations, regular or not.
We set $\pi(\Delta^*)=\mynull$, and we call $\Delta^*$ the \emph{optimal} triangulation of $P$.
Recall that our ordering of the triangulations and thus this notion of optimality relies on the choice of an ordering of the points in $P$.

Algorithm~\ref{alg:down-flip} computes all regular triangulations of $P$.
While it can easily be modified to, e.g., obtain all not necessarily regular triangulations which can be obtained from $\Delta^*$ by down-flips, the computation of the regular triangulations is probably the most interesting use case.
\begin{algorithm}
\begin{algorithmic}[1]
  \Procedure{DFRS}{$P$}
  \State $\Delta \gets$ some regular triangulation of $P$ \label{step:bb}
  \While {$\pi(\Delta) \neq \mynull$}
    \State $\Delta \gets \pi(\Delta)$
  \EndWhile
  \State \Call{RS}{$\Delta,\PhiReg,\pi$} \label{step:reverse}
\EndProcedure
\end{algorithmic}
\caption{Down-flip reverse search}
\label{alg:down-flip}
\end{algorithm}

\newcommand{\maxDeg}{\delta_{\max}}
\newcommand{\maxTriangSpace}{s_{\max}}
\newcommand{\nTriangs}{N}
To analyze Algorithm \ref{alg:down-flip} we introduce three parameters.
First, $\maxDeg$ denotes the maximum degree of a regular triangulation in the flip graph.
This depends on the point configuration $P$ and the ordering of its points, as this defines the optimal triangulation $\Delta^*$.
The following bounds hold:
\begin{equation}\label{eq:maxdeg}
  n-d-1 \ \leq \ \maxDeg \ \leq \ {n \choose d+2} \enspace .
\end{equation}
The lower bound is the dimension of the secondary polytope (cf.\ \cite[Corollary 5.3.2]{Triangulations}), and that binomial coefficient to the right is an upper bound for the number of circuits of $P$.
\begin{remark}
  The upper bound in \eqref{eq:maxdeg} is extremely coarse, and any improvement would be very interesting.
  In practice, the degree seems to be close to the lower bound, which means that a typical secondary polytope is somewhat close to being simple.
\end{remark}
The second parameter, $\maxTriangSpace$, is the maximal number of facets a triangulation of $P$ may have.
It is trivially bounded by
\begin{equation}\label{eq:maxtriangspace}
  \maxTriangSpace \ \leq \ {n \choose d+1} \enspace .
\end{equation}
Yet, if $P$ is a lattice polytope, we obtain a much better bound from the observation that each lattice simplex must have normalized volume at least one.
In this case, this entails
\[
  \maxTriangSpace \ \leq d! \cdot \vol(\conv P) \enspace ,
\]
where $\vol$ is the Euclidean volume.
For instance, this is tight for all cubes $[0,1]^d$.
Finally, the third parameter, $\nTriangs$, is the number of regular triangulations of $P$.
The trivial upper bound
\begin{equation}\label{eq:ntriangs}
  \nTriangs \ \leq \ 2^{\maxTriangSpace} \ \leq \ 2^{n \choose d+1}
\end{equation}
is due to the encoding of a triangulation as the set of its maximal cells.
Note that the upper bound does not benefit from counting regular triangulations only.
That is, the same bound also holds for the number of all triangulations of $P$.

The running time of Algorithm~\ref{alg:down-flip} is dominated by the combined costs for checking the regularity via linear optimization.
Hence it is natural to measure the runtime complexity in the number of LPs to be solved.
In \mptopcom linear programs are solved via \topcom's interface to \cddlib \cite{cddlib}.
This leads to the following straightforward analysis, which generalizes \cite[Theorem 3.6]{AF93} to arbitrary dimensions.
For each fixed ordering of the point set $P$ the beneath-and-beyond method provides the corresponding placing triangulation, which is known to be regular; cf.\ \cite[Lemma 4.3.5]{Triangulations}.
In \mptopcom this serves as the initial triangulation required in Step~\ref{step:bb}.
The following is essentially \cite[Theorem~13]{Imai:2002}.

\begin{theorem}\label{thm:complexity}
  Given some initial regular triangulation, the down-flip reverse search algorithm computes all regular triangulations of~$P$.
  Its running time is bounded from above by the time required to solve
  $
  O(\maxDeg\cdot\nTriangs)
  $
  linear programs in dimension $n$ with at most ${n \choose d}$ linear constraints.
  The space requirement is bounded from above by
  $
  O\left(d \cdot \maxTriangSpace \right)
  $
  plus the space required for solving the linear programs.
\end{theorem}
Here and below we use a unit cost model for representing indices corresponding to points in $P$.
This is adequate since, if $n$ does not fit into a machine size \texttt{Int}, there will be far too many triangulations to allow for any reasonable enumeration.
\begin{proof}
  The correctness follows from (1) the fact that from each regular triangulation we can reach the optimal triangulation $\Delta^*$ via up-flips (cf.\ \cite[Theorem 5.3.2]{Triangulations}), and (2) the correctness of the reverse search method (cf.\ Algorithm~\ref{alg:reverse}).
  
  For the asymptotic running time we can neglect the time for up-flipping from the initial triangulation to $\Delta^*$ since we will revisit all triangulations on that path during the reverse search.
  Each flip from a regular triangulation to another triangulation, regular or not, is considered at most once.
  The total number of such flips is bounded by $\maxDeg\cdot\nTriangs$.
  For each flip we need to decide if the resulting triangulation is regular or not.
  Given a triangulation $\Delta$ its regularity can be determined by solving a linear program in dimension $n$, the number of points.
  Each cell of codimension one gives rise to one linear constraint, and there are not more than ${n \choose d}$ of these.

  The space required comes from storing one triangulation as the set of its maximal cells.
  Each such maximal cell in turn is a subset of the vertices of cardinality $d{+}1$, encoded as a list of integers.
  As pointed out above each of these integers is assumed to be small, whence it is accounted for by a constant space requirement in our analysis.
\end{proof}

By employing \eqref{eq:maxdeg}, \eqref{eq:maxtriangspace} and \eqref{eq:ntriangs} the above complexities can be translated into (horrendous) bounds in terms of the input parameters $d$ and $n$.
Note that, by \cite[Corollary 5.3.11]{Triangulations}, the number of up-flips to the optimal triangulation does not exceed
\[
\min\left\{ (d+2)\binom{n}{\lfloor \frac{d}{2}+1 \rfloor},\, \binom{n}{d+2} \right\} \enspace .
\]
\begin{example} \label{example:moae}
  Figure~\ref{figure:moae_tree} shows the secondary polytope of the ``mother of all examples'' (MOAE) from \cite[Example 5.5.7]{Triangulations}; see also \cite{Pfeifle:MOAE}.
  This is a configuration of six points in $\RR^2$ with three vertices of the convex hull and three points in the interior of the outer triangle.
  There are 18 triangulations, 16 of which are regular.
  The 18 triangulations are grouped into five orbits, and the coloring of the vertices in Figure~\ref{figure:moae_tree} shows those orbits.
  The two nonregular triangulations (are sub-regular and) share the same GKZ-vector, and these occur as the central blue point in a hexagon with yellow vertices.
  The optimal triangulation corresponds to the one black vertex (at the top and toward the back).
  The thick edges form the reverse search tree of the graph $\PhiRegComp=\Phi$; the optimal triangulation is the root.
\end{example}
\begin{figure}[htb]
\includegraphics[scale=.5]{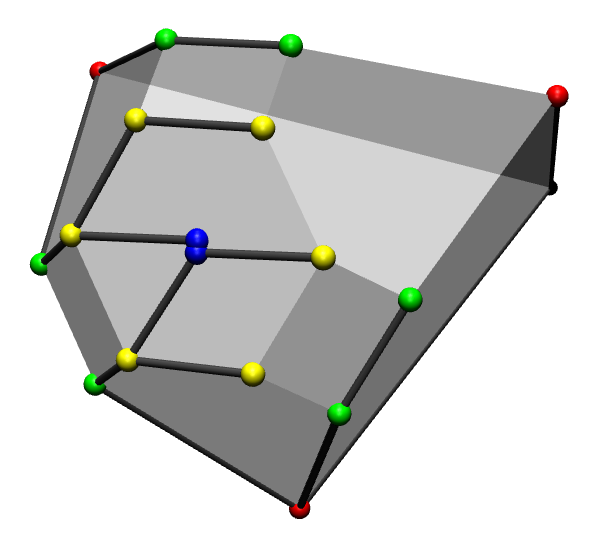}
\caption{Reverse search tree of MOAE from Example~\ref{example:moae}.  The two blue points correspond to the two nonregular triangulations.  Since they share the same GKZ-vector they actually coincide; here we chose to draw them slightly apart in order to reveal the tree.}
\label{figure:moae_tree}
\end{figure}
\begin{remark}
  A triangulation of $P$ is \emph{full} if it uses all the points in $P$.
  It was proved in \cite[\S5]{Imai:2002} that the subgraph of $\PhiReg$ induced on the full triangulations is connected; see also \cite{PourninLiebling:2007} and \cite[Corollary 5.3.14]{Triangulations}.
  As a consequence, the down-flip reverse search algorithm can be applied to enumerate the full triangulations only.
\end{remark}

\section{Triangulations up to symmetry}
\label{sec:symmetry}
\noindent
In many applications for which the set of all (regular) triangulation is sought after the point set exhibits a great deal of symmetry.
Typical examples are the set of vertices of a high-dimensional cube (cf.\ Section~\ref{subsec:cube}) or the set of lattice points in a dilated simplex (cf.\ Section~\ref{subsec:dilated-simplex}).
It is natural to exploit this symmetry, and this is a standard feature of \topcom and \mptopcom.
An obvious drawback of enumerating via reverse search is that the group of, say, affine automorphisms of the point set does not operate on the reverse-search tree.
That is, applying reverse search up to symmetry requires some extra considerations; see Bremner, Dutour Sikiri\'c and Sch\"urmann \cite[\S7]{BremnerDutourSchuermann:2009} for a brief discussion in the context of convex hull computations.
The simple idea is to apply the reverse search scheme to the graph whose nodes correspond to the orbits of the (possibly regular) triangulations and whose edges are induced by the edges in $\Phi$, i.e., by flips.

Following Imai et al.~\cite{Imai:2002} we suggest to adapt the approach via GKZ-vectors from Section~\ref{sec:triangulations} to the symmetric setting.
Let $P\subset\RR^d$ be a finite point set, and let
\[
\group \ \leq \ \text{SL}_d(\RR) \rtimes \RR^d
\]
be a finite group of affine unimodular automorphisms which acts on the set $P$ by permutations.
In particular, this action is faithful, and it preserves the volume.
Clearly, there is an induced action of $\group$ on the set of all triangulations of $P$, which leaves the set of regular triangulations invariant.
\begin{lemma}\label{lemma:gkzaction}
  Let $\groupElem$ be an element of\/ $\group$, and let $\Delta$ be a triangulation.
  Then we have
  \[
  \gkz_{\groupElem\cdot\Delta}(\groupElem(p)) \ = \ \gkz_\Delta(p)
  \]
  for all points $p\in P$.
\end{lemma}
\begin{proof}
 This is an immediate consequence of $\group$ preserving the volume.
\end{proof}

The main challenge in implementing reverse search for enumerating triangulations is to find a fast implementation for computing the canonical representatives.
To this end we will to some extent deviate from standard wisdom in the algorithmic theory of permutation groups as presented, e.g., in the monograph \cite{Seress:2003} by Seress.
One reason is that our groups are fairly small.
In most relevant cases $\group$ contains a few thousand elements, while the set of triangulations on which $\group$ acts often has millions of orbits.
Moreover, despite the fact that the point set is symmetric, few triangulations exhibit much symmetry.
This entails that most orbits are about the size of the entire group.
See Section~\ref{sec:experiments} for more details and precise numbers.

In order to employ reverse search for traversing (a connected component of) the flip graph, the total ordering on the set of triangulations introduced in Definition~\ref{def:total} is crucial.
This leads us to represent the $\group$-orbit of a triangulation $\Delta$ of $P$ by
\begin{equation}\label{eq:can}
  \canonical(\Delta) \ := \ \max (\group\cdot\Delta) \enspace .
\end{equation}
That is, the \emph{canonical representative} $\canonical(\Delta)$ of an orbit $\group\cdot\Delta$ is characterized as follows:
\begin{inparaenum}
\item its GKZ-vector is lexicographically maximal among all GKZ-vectors of the triangulations in $\group\cdot\Delta$, and
\item among all triangulations in $\group\cdot\Delta$ which satisfy (i) it is lexicographically maximal, considered as a characteristic vector of maximal simplices.
\end{inparaenum}

In particular, $\canonical$ maps triangulations to triangulations in the same orbit such that $\canonical(\Delta)=\canonical(\Delta')$ if and only if $\group\cdot\Delta=\group\cdot\Delta'$, and it follows that $\canonical(\canonical(\Delta))=\canonical(\Delta)$.
Recall that the above definition still depends on the ordering of the points in $P$ as well as on the ordering of the maximal simplices.
If the triangulation is regular, then the entire orbit $\group\cdot\Delta$ consists of regular triangulations and thus their GKZ-vectors are pairwise distinct.

The following example shows that the GKZ-vectors of nonregular triangulations may behave in an unexpected way, and this shows that condition (ii) is necessary.
\begin{example}\label{exmp:cube}
  Let $I^4$ be the set of 16 vertices of the 4-dimensional $0/1$-cube.
  Here and below we use $I=[0,1]$ to denote the unit interval on the real line.
  We may read each vertex as a bitstring of length four, and in this way we obtain a natural encoding in terms of the hexadecimal digits $\mathtt{0}$ through $\mathtt{F}$.
  The full group $\group$ of affine unimodular automorphisms is the wreath product of a cyclic group of order two (corresponding to a reflection which sends $x_i$ to $1-x_i$) by the symmetric group $\Sym(4)$ (acting on the four neighbors of a fixed vertex); the total size of $\group$ equals $384$.
  Consider the triangulation $\Delta$ whose maximal simplices read
  \[
  \begin{array}{ccccccccc}
    \mathtt{01278} & \mathtt{0157D} & \mathtt{0178D} & \mathtt{02478} & \mathtt{0457D} & \mathtt{0478D} &
    \mathtt{1237A} & \mathtt{1278A} & \mathtt{137AB}\\
    \mathtt{178AB} & \mathtt{178BD} & \mathtt{189BD} & \underline{\mathtt{2467C}} & 
    \underline{\mathtt{2478C}} & \underline{\mathtt{267AC}} & \underline{\mathtt{278AC}} &
    \mathtt{478CD} & \mathtt{67ACE}\\
    \mathtt{78ABC} & \mathtt{78BCD} & \mathtt{7ABCE} & \mathtt{7BCDF} & \mathtt{7BCEF} \\
  \end{array}
  \]
  There is a flip, $f=[\Delta\rightsquigarrow\Delta']$, to another triangulation, $\Delta'$, which replaces the underlined simplices with
  \[
  \begin{array}{cccc}
    \mathtt{2467A} & \mathtt{478AC} & \mathtt{467AC} & \mathtt{2478A}
  \end{array}.
  \]
  The triangulations $\Delta$ and $\Delta'$ lie in different $\group$-orbits (both of which have the maximal length 384), and their respective GKZ-vectors are
  \begin{align*}
    &(6,10,8,2,6,2,3,23,14,1,9,10,11,10,3,2)
    \intertext{and}
    &(6,10,6,2,8,2,3,23,14,1,11,10,9,10,3,2) \enspace.
    \intertext{Yet the lexicographically maximal GKZ-vector in both orbits is the same, and it reads}
    &(23,3,2,8,2,6,10,6,2,3,10,9,10,11,1,14) \enspace.
  \end{align*}
  It follows that neither $\Delta$ nor $\Delta'$ are regular (but they turn out to be sub-regular).
  The same holds for the canonical representatives $\canonical(\Delta)$ and $\canonical(\Delta')$.

  This example illustrates that, in general, the GKZ-vectors cannot distinguish between $\group$-orbits.
  Moreover, it also shows that GKZ-vectors alone do not suffice to establish a total ordering on the flip graph.
  For a tie-breaker we need, e.g., the lexicographic ordering as in Definition~\ref{def:total}. 
%
\end{example}

\section{Canonical representatives via switch tables}
\label{sec:canonical}
\noindent
The purpose of this section is to explain why the canonical representatives of triangulations are chosen as in~\eqref{eq:can} and how this can be exploited.
To this end we start out with a more abstract setting.
We assume that the finite group $\group$ acts on the set $[m]:=\{0,\ldots,m-1\}$.
The following concept is our main tool.
\begin{definition}\label{def:switch}
  An \emph{$m$-switch table} for $\group$ is a function
  \[
  \begin{array}{cccl}
    \switch: & [m]\times[m] & \to & \group\\
    & (i,j) & \mapsto &
    \begin{cases}
      \groupElem\in \group,\mbox{ with } \groupElem(k)=k \mbox{, for } k<i \mbox {, and } \groupElem(j)=i & \mbox{if it exists}\\
      \id & \mbox{otherwise} \enspace .
    \end{cases}
  \end{array}
  \]
\end{definition}
Note that we have $j>i$ if $\switch(i,j)\neq\id$. We denote by $\switchTableSize(\switch)$ the \emph{depth} of the switch table, which is defined as
\[
\switchTableSize(\switch)\ :=\ \max\SetOf{i+1}{\text{there is an index } j \text{ with } \switch(i,j)\not=\id}\enspace .
\]
That is, $\switchTableSize$ is the index of the first row of $\switch$ which only contains the identity.
In general, a switch table is by no means unique as there may be many candidates in $\group$ for $\switch(i,j)$.

The $i$th row of a switch table tells us which elements $\ge i$ can be moved to position $i$ while leaving the first $i$ elements of $[m]$ unchanged.
That is, the switch $\switch(i,j)$ lies in the stabilizer
\[
\group_{[i]} \ := \ \SetOf{g\in\group}{g(j)=j \text{ for all } j<i} \enspace .
\]
Formally, we have $\group_{[0]}=\group$.
Moreover, the stabilizer $\group_{[m]}$ is the kernel of the action of $\group$ on the set $[m]$, and this is a normal subgroup.
In particular, that action is faithful if and only if $\group_{[m]}$ is the trivial group.
Let
\[
\sigma_i \ := \ 1 + \# \SetOf{\switch(i,j)\neq\id}{j>i}
\]
be one plus the number of nontrivial entries in the $i$th row of the switch table.
\begin{proposition}\label{prop:transversal}
  The $i$th row of the switch table forms a (left) transversal of the subgroup $\group_{[i+1]}$ in~$\group_{[i]}$.
  In particular, the number $\sigma_i$ is the index of $\group_{[i+1]}$ in $\group_{[i]}$.
\end{proposition}
\begin{proof}
  Consider the case $i=0$, where we need to show that the $0$th row of the switch table gives a transversal of $H:=\group_{[1]}$ in $\group=\group_{[0]}$.
  The group $\group$ acts transitively on the left cosets of $H$ by multiplication on the left.
  Let $g$ be an element in the complement $\group\setminus H$.
  If this does not exist there is nothing to show.
  Since $g$ does not stabilize $i=0$ there is an index $j>0$ such that $g\cdot j = 0$.
  By definition of the switch table there must also be a switch $s=s(0,j)$ with $s\cdot j = 0$.
  We obtain that $s^{-1}g$ lies in $H$ and thus $sH=gH$, which proves the claim for $i=0$.

  Erasing the $0$th row and the $0$th column of the switch table yields a switch table for $H=\group_{[1]}$.
  This shows that, inductively, the above argument also resolves the general case.
\end{proof}

\begin{corollary}\label{cor:uniqueness}
  For each element $g\in\group$ there is a unique sequence $\switch(0,j_0),\switch(1,j_1),\dots,\switch(m-1,j_{m-1})$ of switches such that the product
  \[
  \switch(0,j_0)\switch(1,j_1)\cdots\switch(m-1,j_{m-1})g^{-1}
  \]
  lies in the kernel $\group_{[m]}$.
\end{corollary}
\begin{proof}
  From Proposition~\ref{prop:transversal} we know that there is a unique switch $\switch(0,j_0)$ with $\switch(0,j_0)^{-1}g\in\group_{[0]}$.
  Inductively, we obtain $\switch(i+1,j_{i+1})$ by requiring \[ \switch(i+1,j_{i+1})^{-1}\switch(i,j_i)^{-1}\cdots\switch(0,j_0)^{-1}g\in\group_{[i+1]} \enspace.\]
  The uniqueness follows as we have $\switch(i,j)=\id$ whenever $j\leq i$.
\end{proof}

Let $\sigma := \sigma_0\sigma_1\cdots\sigma_{m-1}$ be the product of the $\sigma_i$.

\begin{corollary}
  We have $\sigma\leq m!$, and equality holds if and only if the quotient $\group/\group_{[m]}$ is the symmetric group $\Sym(m)$.
  Moreover, $\sigma$ agrees with the order of $\group$ if the action of $\group$ on $[m]$ is faithful.
\end{corollary}
\begin{proof}
  We have $\sigma_i\leq m-i$, and this shows the upper bound.
  By Corollary~\ref{cor:uniqueness} the number $\sigma$ is the index of the kernel $\group_{[m]}$ in $\group$.
  The quotient $\group/\group_{[m]}$ is isomorphic to a subgroup of $\Sym(m)$, and this characterizes the equality case.
\end{proof}

A switch table is somewhat reminiscent of a system of strong generators for $\group$; cf.\ \cite[\S4]{Seress:2003}.
These are used, e.g., in \permlib \cite{permlib}.
However, the concepts differ as can be seen from the example below.

\begin{example}\label{exmp:switch}
  Let $\group=\Sym(4)$ be the symmetric group of degree four acting naturally on the set $\{0,1,2,3\}$.
  Using cycle notation a switch table $\switch$ is given by
  \[
    \begin{array}{lll}
      \switch(0,1) = (0\ 3\ 2\ 1) \,, & \switch(0,2) = (0\ 2)(1\ 3) \,, & \switch(0,3)=(0\ 1\ 2\ 3) \,,\\
      \switch(1,2) = (1\ 3\ 2) \,, & \switch(1,3) = (1\ 2\ 3) \,, & \switch(2,3)=(2\ 3) \,,
    \end{array}
  \]
  and all other entries of the switch table are the identity element. The depth of this switch table is $3$.
  We have $\sigma_0=4$, $\sigma_1=3$, $\sigma_2=2$ and $\sigma=24$, which is the order of $\Sym(4)$.

  Let $\group_i^{\switch}$ be the subgroup of $\group$, called the \emph{$i$th switch group}, which is generated by the $i$th row of the switch table $\switch$.
  The switch group $\group_0^{\switch}$ is generated by the $4$-cycle $(0\ 1\ 2\ 3)$, while $\group_1^{\switch}$ is generated by the $3$-cycle $(1\ 2\ 3)$, and $\group_2^{\switch}$ is generated by the transposition $(2\ 3)$.
  Consequently, $\group_0^{\switch}$ is a proper subgroup of $\group_{[0]}=\group$.
  Similarly, $\group_1^{\switch}$ is a proper subgroup of $\group_{[1]}$, which is the symmetric group of degree three acting on $\{1,2,3\}$,  whereas $\group_2^{\switch}$ agrees with $\group_{[2]}$.
  Observe that $\group_2^{\switch}$ is not a subgroup of~$\group_1^{\switch}$, which in turn is not a subgroup of~$\group_0^{\switch}$.
\end{example}

Now we consider a second action of $\group$ on some other set $\Omega$ and a map $\eval:\Omega\to\RR^m$.
Our first action of $\group$ on $[m]$ induces a (linear) action on the vector space $\RR^m$ by permuting the coordinates.
If the compatibility condition 
\[
\eval(g \cdot \omega) \ = \ g \cdot \eval(\omega)
\]
is met we call $\eval$ an \emph{$m$-evaluation map} of the $\group$-action on the set $\Omega$.
Due to the compatibility we obtain an action of $\group$ on the set $\eval(\Omega)\subset\RR^m$.
Now we define the \emph{canonical representative} $\canonical(\eval(\omega))$ for that action as the lexicographically maximal vector in the orbit $\group \cdot \eval(\omega)$.
Below we will discuss the relationship of this definition with the canonical representative of a triangulation from~\eqref{eq:can}.
The following example is natural.
\begin{example}\label{ex:characteristic}
  If $\group$ acts on $[m]$ then this induces a second action on the set $\Omega$ of all subsets of $[m]$.
  Sending $\omega\in\Omega$ to its characteristic vector of length $m$ yields an evaluation map, say~$\chi$.
  The canonical representative of an orbit is the lexicographically smallest set in that orbit.

  This constitutes the second part of our approach for $\mptopcom$. The
  two-step approach can be summarized as choosing $\eval$ to be
  $\eval:=(\gkz,\chi)$.
  The canonical representative then has lexicographically largest
  $\gkz$-vector, and is the largest element among those with the same
  $\gkz$-vector. If one is looking at regular triangulations exclusively, one
  may as well forget about the second part and choose $\eval:=\gkz$.
\end{example}

Our main method for computing canonical representatives is the procedure \textproc{canonical} in
Algorithm~\ref{alg:canonical-rep}.  It essentially relies on the function \textproc{GoodSwitches} in
Algorithm~\ref{alg:good_switches} which determines all switches that may lexicographically improve a
given vector $z\in\RR^m$ in the orbit of $\group_{[i]}$. The idea is to employ a depth first search.

\begin{algorithm}[thb]
\begin{algorithmic}[1]
  \Procedure{GoodSwitches}{$z,i,\switch$}
   \State $Y = (y_0,y_1,\dots,y_{\ell-1}) \gets $ sort descending $\SetOf{z_j}{z_j>z_i}$
   \For {$k=0,1,\dots,\ell-1$}
      \State $J \gets \SetOf{j\in[m]}{z_j=y_k}$
      \State $S \gets \SetOf{\switch(i,j)}{j\in J,\ \switch(i,j)\neq\id }$
        \If { $S \not= \emptyset$ }
            \State \Return $S$ \label{step:first-kind}
        \EndIf
   \EndFor
   \State $J \gets \SetOf{j\in[m]}{z_j=z_i}$
   \State \Return $\SetOf{\switch(i,j)}{j\in J}$
\EndProcedure
\caption{Find switches which may yield a larger vector in $\group_{[i]}\cdot z$}
\label{alg:good_switches}
\end{algorithmic}
\end{algorithm}

\begin{algorithm}[bht]
\begin{algorithmic}[1]
  \Procedure{canonical}{$\omega, \eval, i, \switch$}
  \If{$i>\switchTableSize(\switch)$}
      \State\Return $\omega$
  \EndIf
  \State $S \gets$\Call{GoodSwitches}{$\eval(\omega),i, \switch$}
  \State $\omega' \gets \omega$
  \For { $s \in S$ }
    \State $\omega'' \gets$ \Call{canonical}{$s \cdot \omega, \eval, i+1, \switch$}
    \If {$\eval(\omega'')>\eval(\omega')$}
      \State $\omega' \gets \omega''$
    \EndIf
  \EndFor
  \State \Return $\omega'$
\EndProcedure
\caption{Canonical representative of an orbit}
\label{alg:canonical-rep}
\end{algorithmic}
\end{algorithm}

\begin{proposition}\label{prop:can}
  Let $\omega\in\Omega$ be an arbitrary element.
  For all $i\in[m]$, Algorithm \ref{alg:canonical-rep} computes an element $\omega'\in\group_{[i]}\cdot\omega$ such that $\eval(\omega')$ is lexicographically maximal among all elements in the orbit $\group_{[i]}\cdot\omega$.
  In particular, for $i=0$, the evaluation $\eval(\omega')$ is the canonical representative $\canonical(\eval(\omega))$.
\end{proposition}

\begin{proof}
  The correctness of Algorithm \ref{alg:canonical-rep} essentially follows from Corollary~\ref{cor:uniqueness}, which provides a transversal of $\group_{[m]}$ as a subgroup of $\group$ in terms of switches.
  As that transversal arises from an inductive construction Algorithm \ref{alg:canonical-rep} works recursively.

  It remains to discuss the case where $\eval(\omega')_0=\eval(\omega)_0$.
  Then we need to consider all possible switches which may or may not change the entry with index $0$ but keep the $\eval$-value.
  This may include the identity.
\end{proof}

\begin{example}\label{exmp:can}
  Again we consider the symmetric group of degree four, $\group$, acting on the set $\{0,1,2,3\}$ and the $4$-switch table $\switch$ from Example~\ref{exmp:switch}.
  As a second action of $\group$ we take the induced action on the powerset $\Omega=2^{\{0,1,2,3\}}$.
  Then the map $\eval$ which sends $\omega\in\Omega$ to its characteristic vector is a $4$-evaluation; see Example~\ref{ex:characteristic}.
  We want to compute $\canonical(\{1,2\},0)$ via Algorithm~\ref{alg:canonical-rep}, i.e., we are taking the two-element set $\omega=\{1,2\}$ as our input.
  Hence we have
  \[
  \eval(\omega) \ = \ (0,1,1,0) \enspace.
  \]
  The set $S$ originates from Step~\ref{step:first-kind} in Algorithm~\ref{alg:good_switches}, and it comprises the two permutations $(0\ 3\ 2\ 1)$ and $(0\ 2)(1\ 3)$.

  In the first round of the for-loop we have $\switch(0,1)=(0\ 3\ 2\ 1)$, and thus $\switch(0,1)\cdot\omega=\{0,1\}$.
  Recursively we compute $\canonical(\{0,1\},1)=\{0,1\}$.
  For the second switch $\switch(0,2)=(0\ 2)(1\ 3)$ with $\switch(0,2)\cdot\omega=\{0,3\}$ we get $\canonical(\{0,3\},1)=\{0,1\}$.
  Hence, both branches return $\{0,1\}$ and the algorithm produces $\{0,1\}$ as its output.
\end{example}

For $\omega\in\Omega$ denote by $\jbound(\omega)$ the maximum number of identical entries of the vector $\eval(\omega)$.
Since $\group$ acts on the set $\eval(\Omega)$ by coordinate permutations, the number $\jbound(\omega)$ is an invariant of the orbit $\group\cdot\omega$.
We let $\jbound$ be the maximal $\jbound(\omega)$, taken over all elements in $\Omega$.

Define $\phi_i$ as the minimum of $\jbound$ and $\sigma_i$; this depends both on the switch table and on the evaluation function.
Furthermore, we set $\phi:=\phi_0\phi_1\cdots\phi_{\switchTableSize(\switch)-1}$.
By construction we have $\phi\leq\sigma\leq\#\group$ and $\phi\leq m!$.
The benefit from our somewhat elaborate setup comes from the fact that the number $\phi$ may be much smaller than the order of the group $\group$; see Example~\ref{exmp:I4} below.

\begin{corollary}\label{cor:complexity}
  The worst-case time complexity of Algorithm~\ref{alg:canonical-rep} is of order $O(\max\{\phi\cdot m, m\log m\})$ and the worst-case space complexity is of order $O(m^3)$.
\end{corollary}
\begin{proof}
   \textit{Time complexity:}
  In the procedure \textproc{GoodSwitches}, called with the vector $z$ of length $m$ and the index $i$, the set $E$ constructed in Step~2 has size at most $m-1$.
  Hence sorting $E$ is of order $O(m\log m)$.
  The size of set $J$ constructed in Step~5 or Step~11 is bounded by the number of identical entries $\jbound(\omega)$ in the vector $\eval(\omega)$. Thus, the set $S$ constructed in Step~6 or Step~12 is bounded by the size of $J$ as well. But the set $S$ also contains only non-trivial entries from the $i$th row of the switch table, in case of Step~12 a single copy of the identity. Hence
  the set of switches returned is of size at most $\phi_i$.
  
  The height of the recursion tree of \textproc{canonical} is at most $m-1$, and hence the number of leaves is bounded by $\phi_0\phi_1\cdots\phi_{m-1}=\phi$.
  Consequently, the total number of nodes and also the total number of edges is of order $O(\phi)$.
  For each edge we lexicographically compare two vectors of length $m$.
  All other costs are dominated by the total complexity of these comparisons.
   
  \textit{Space complexity:} The number of entries in the switch table is given by the expression
  \begin{equation}\label{eq:complexity:space}
    \sum_{i=0}^{\switchTableSize(\switch)-1}\phi_i \ \leq \ \sum_{i=0}^{\switchTableSize(\switch)-1}m-i \ \leq \ m^2 \enspace .
  \end{equation}
  Algorithm~\ref{alg:good_switches} returns at most one row of the switch table, and the recursion depth is bounded by the depth of the switch table.
  Each element of the switch table is a permutation of length $m$.
  Thus the total space requirement amounts to $O(m^3)$.
\end{proof}
Note that $O(m^3)$ is a coarse estimate.
In practice the depth of the switch table is often small, and then \eqref{eq:complexity:space} provides better bounds.

\medskip

Finally, we can explain our choice for the canonical representative of a triangulation as in \eqref{eq:can}, which rests on Definition~\ref{def:total}.
Let us recall our setup.
The point set $P\subset\RR^d$, of cardinality $n$, is affinely spanning.
It is equipped with a group $\group$ of affine unimodular automorphisms.
This action induces an action on the set $\Omega$ of all triangulations of $P$.
We can encode a triangulation as its set of $d$-simplices, which are the maximal cells.
The latter are encoded as those $(d{+}1)$-subsets of $P$ which are affinely spanning.
If $m$ is the number of all $d$-simplices then encoding a triangulation $\Delta$ as its characteristic vector among the set of all $d$-simplices yields an $m$-evaluation map of the action of $\group$ on $\Omega$; see Example~\ref{ex:characteristic}.
For computing the canonical representative of a triangulation $\Delta\in\Omega$ in a brute-force approach all elements of $\group$ are applied and the lexicographically largest one is picked.
Of course, the elements of $\group$ can be precomputed once in the initialization.
Since most orbits are expected to be about the size of $\group$ this is often superior to the more traditional approach of trying generators of $\group$ until no new triangulations are found, since it requires fewer comparisons of (characteristic vectors of) triangulations.

Yet evaluating at GKZ-vectors leads to a significant improvement.
The map which sends a triangulation $\Delta$ to its GKZ-vector is an $n$-evaluation, and $n$ is always much smaller than $m$.
More importantly most entries of a typical GKZ-vector are distinct and hence the parameter $\jbound$ which enters the complexity analysis in Corollary~\ref{cor:complexity} is very small.
This explains criterion (i) in Definition~\ref{def:total}.
The tie-breaker criterion (ii) is only necessary for dealing with nonregular triangulations.
The following example occurs in a computation which should be considered small by current standards.
Larger input results in larger gains. 
\begin{example}\label{exmp:I4}
  Let us again look at the set $P=I^4$ of vertices of the $4$-dimensional regular cube.
  The group $\group$ is the full group of affine (unimodular) automorphisms of order 384.
  Mapping to GKZ-vectors yields a $16$-evaluation.
  Any switch table has depth four with
  \[
    \sigma_0 = 16 \,,\ \sigma_1 = 4 \,,\ \sigma_2 = 3 \,,\ \sigma_3 = 2 \enspace ,
  \]
  Their product $\sigma=\sigma_0\sigma_1\sigma_2\sigma_3$ equals $384$, which is the order of the group $\group$.

  To assess the complexity of computing the canonical representatives, the shape of the GKZ-vectors is the key.
  Here is the full distribution of values of $\jbound(\cdot)$ for the 247\,451 orbits of sub-regular triangulations of $I^4$.
  \begin{center}
    \begin{tabular}{cccccccccccc}
      \toprule
      1 & 2     & 3      & 4     & 5     & 6    & 7   & 8   & 9  & 10 & 11 & 12 \\
      \midrule
        & 38\,673 & 134\,773 & 58\,835 & 11\,699 & 2\,985 & 364 & 107 & 11 & 2  &    & 2 \\
      \bottomrule
    \end{tabular}
  \end{center}
  The average value of $\jbound$ (taken over all orbits) is approximately $3.22$.
  This results in an average complexity parameter of $\phi \approx 66.8$, which is much smaller than $384$, the order of the group.
\end{example}

\section{Implementation details}
\label{sec:implementation}
\noindent
Our implementation \mptopcom builds on and uses existing code from \mts \cite{mts_tutorial}, \polymake \cite{polymake} and \topcom \cite{TOPCOM}.
The \mts setup dedicates one process to the master and a second one for dealing with the output; the remaining processes are reserved for the workers.

Let us give a brief overview of the interplay of the different software
systems: While parallelization is handled by \mts, the single worker processes
are in the domain of \polymake and \topcom. Triangulations and groups are
handled by \topcom's highly optimized code. Since \topcom provides vectors and
matrices already, a first implementation used these for $\gkz$-computation.
However, replacing \topcom's vectors and matrices by the corresponding objects
from \polymake vastly improved performance. At this point, all vectors,
matrices, and sets are handled by \polymake. Checking regularity is handled by
\topcom's internal interface to \cddlib \cite{cddlib}.

Parallelizing with more threads will scale almost linearly at first, but depending on
the size of the example, the curve of time consumption over number of threads
will flatten sooner or later, as can be seen in the Figures
\ref{fig:cube-timings-ryzen}, \ref{fig:cube-timings-cluster}, and
\ref{fig:s3s7-timings-cluster}. In order to determine whether a down-flip is
valid, the reverse search algorithm descends to the target node and checks whether
the predecessor of the target is the original node. This happens essentially
for every down-flip into a triangulation. Thus, for a given triangulation the
predecessor is computed multiple times, and hence, its neighbors are computed
multiple times. If one considers regular triangulations exclusively, then
regularity is checked multiple times as well. In the symmetric case, the
canonical representative is computed several times. We attack this problem by
maintaining three caches.  The key type in every cache is a triangulation
$\Delta$, so we just list the values.
\begin{enumerate}
\item \emph{Flip cache:} Contains the list of all flips of $\Delta$.
\item \emph{Orbit cache:} The canonical representative from the orbit of $\Delta$.
\item \emph{Regularity cache:} A boolean whether $\Delta$ is regular or not.
\end{enumerate}
Each worker is equipped with three such caches; they follow the
least-recently-used paradigm and can store a fixed number of keys, subject to change by the
user.
Caching is combined with hashing such that previously computed data is instantly available, if it is still cached.
There are a number of further caches with dynamic size, e.g. a cache containing
the volumes of all maximal simplices. These caches are filled at startup before
the reverse search begins and never changed again.
Sharing these caches among the workers as hinted at before
is possible, but not realized yet, and would probably damage the flexibility of
\mptopcom.

Parallelizing does not change the amount of down-flips going into a
triangulation. But since the caches are not shared among the workers, every
worker populates its own caches, leading to the non-linearity in scaling. In
theory, working without caches would scale linearly, but is infeasible for
larger examples, even starting with $I^4$.

We use several encodings of the same triangulation simultaneously.
For instance, enumerating all maximal simplices spanned by our point configuration first allows to store a triangulation as a set of machine-size integers.
Those integers are the indices pointing into the array of maximal simplices.

The implementation is flexible enough to deal with several scenarios.
In our fastest setup we assume that the coordinates of the point configuration are machine size integers.
This means that we may use \texttt{int} for the entries of the GKZ-vectors.
The condition on the integrality is natural for the applications for algebraic and tropical geometry we have in mind.

As pointed out several times the total ordering on the triangulations from Definition~\ref{def:total} depends on an ordering of the points.
In practice it is often beneficial to pick a random ordering.
This may reduce the height of the search tree for the canonical representative.

We experimented with several compilers on various kinds of hardware.
By and large we found \texttt{clang}, version 3.8.0, to be about 10\% faster than various versions of \texttt{gcc} compilers.
Therefore our timings below employ \texttt{clang}.
For the hardware we tried the following:
\begin{itemize}
\item A cluster with four nodes, each of which is equipped with 2 $\times$ Intel Xeon E5-2630 v2 Hexa-Core (2600--3100MHz, 5201.45 bogomips) and 64GB RAM per node.
  On this machine we used 40 threads. 
The operating system is openSUSE 42.2 with kernel version 4.4.79.
\item A desktop machine with an AMD Ryzen 7 1700 CPU.  This has 8 cores/16 threads (3000 MHz, 5967.87 bogomips) and 32GB RAM.
The operating system is openSUSE 42.1 with kernel version 4.12.1.
\end{itemize}
Recall that two processes are reserved for the master and the output processes.
That is, $n$ processes means $n-2$ workers.

\section{Experimental results}
\label{sec:experiments}
\noindent
We tried our implementation on a number of point configurations which occur naturally in geometric combinatorics and related areas.
Our notation is as follows.
The point set $\Delta_d$ comprises the $d+1$ vertices $0,e_1,e_2,\dots,e_d$ of the standard $d$-dimensional simplex, where $e_i$ is the $i$th standard basis vector of $\RR^d$.
We abbreviate $I=[0,1]$, and so $I^d$ is the $d$-dimensional unit cube.
From these several interesting point configurations can be formed, e.g., by taking products.
The point configurations of type $I^d$ or $\Delta_p\times\Delta_q$ or $I^p\times\Delta_q$ are in convex position, i.e., they form the vertices of polytope.
For such point configurations Table~\ref{tab:summary} shows the number of orbits of regular and sub-regular triangulations; in all cases the group $\group$ is the full group of affine symmetries.

To the best of our knowledge the results for $\Delta_2\times\Delta_6$, $\Delta_3\times\Delta_4$ and $3\cdot\Delta_3$ are new.

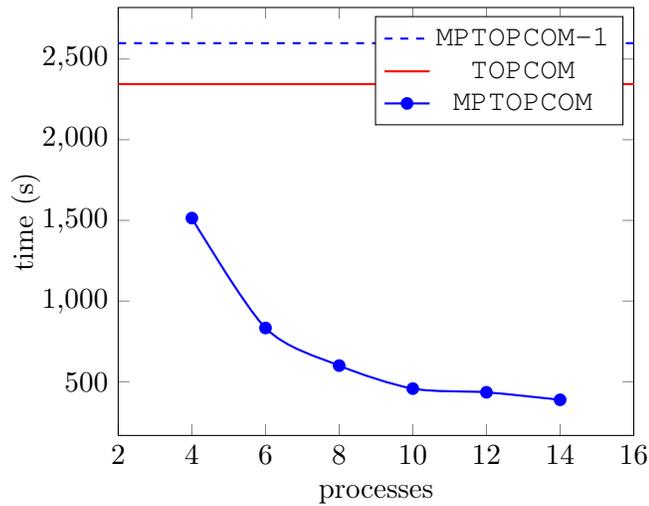
\begin{figure}[htb]
\begin{tikzpicture}
\begin{axis}[
      ylabel=\text{time (s)},
      xlabel=\text{processes},
      xmin = 2, xmax = 16,
   ]
   \addplot[mark=none, blue, thick, dashed, samples = 2, domain=2:100] {2598.29};
   \addlegendentry{\mptopcomone}
   \addplot[mark=none, red, thick, samples = 2, domain=2:100] {2344.61};
   \addlegendentry{\topcom}
   \addplot[smooth, mark=*, blue, thick] plot coordinates {
      (4, 1514.21)
      (6, 833.09)
      (8, 600.08)
      (10, 457.35)
      (12, 434.37)
      (14, 387.97)
   };
   \addlegendentry{\mptopcom}
\end{axis}
\end{tikzpicture}
\caption{Timings for enumerating the sub-regular triangulations of $I^4$ taken on AMD Ryzen 7 1700 with 32GB RAM, depending on the number of processes.
  The timings for the single-threaded version of \mptopcom (marked with ``\texttt{-1}'') and \topcom are added for reference.}
\label{fig:cube-timings-ryzen}
\end{figure}

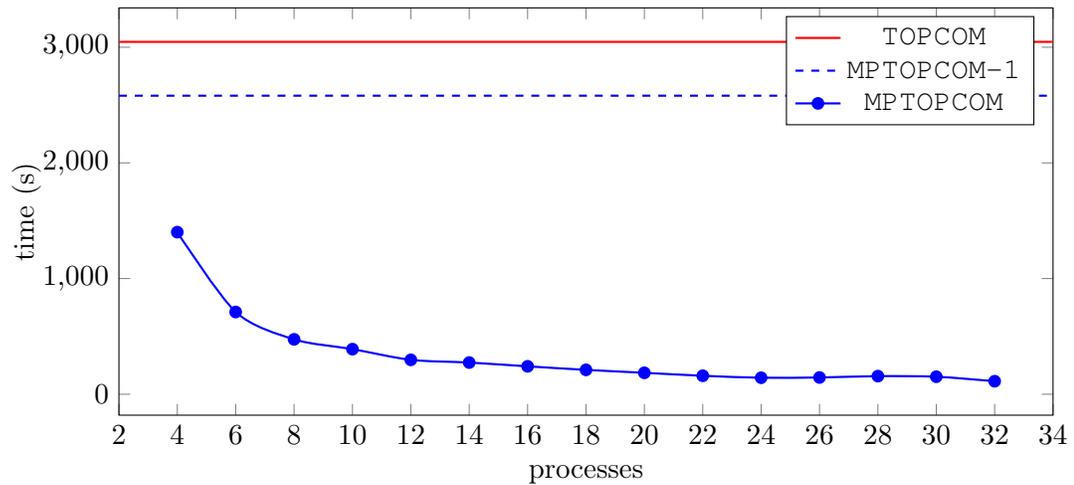
\begin{figure}[htb]
\begin{tikzpicture}
\begin{axis}[
      ylabel=\text{time (s)},
      xlabel=\text{processes},
      xmin = 2, xmax = 34,
      width = 14cm, height = 7cm
   ]
   \addplot[mark=none, red, thick, samples = 2, domain=2:100] {3046.26};
   \addlegendentry{\topcom}
   \addplot[mark=none, blue, thick, dashed, samples = 2, domain=2:100] {2580.70};
   \addlegendentry{\mptopcomone}
   \addplot[smooth, mark=*, blue, thick] plot coordinates {
      (4, 1401.14)
      (6, 710.93)
      (8, 474.37)
      (10, 389.68)
      (12, 297.19)
      (14, 273.82)
      (16, 241.14)
      (18, 210.64)
      (20, 185.16)
      (22, 159.67)
      (24, 142.50)
      (26, 144.82)
      (28, 156.45)
      (30, 151.29)
      (32, 112.38)
   };
   \addlegendentry{\mptopcom}
\end{axis}
\end{tikzpicture}
\caption{Timings for enumerating the sub-regular triangulations of $I^4$ taken on the Intel Xeon E5-2630 v2 cluster with 64GB RAM per node, depending on the number of processes.}
\label{fig:cube-timings-cluster}
\end{figure}

\subsection{The $4$-cube}
\label{subsec:cube}
A standard test case is the four-dimensional cube $I^4$; see also Example~\ref{exmp:I4}.
The group of affine unimodular automorphisms has order 384, and there are 247\,451 orbits of triangulations in the flip-component of the regular triangulations.
By today's standard this is to be considered small input.
Out of these, actually $235\,277$ are regular \cite[Thm.~6.3.12]{Triangulations}.
Pournin proved that the flip graph of $I^4$ is connected \cite{Pournin:2013}.
Since \mptopcom also counts $247\,451$ orbits it follows that each triangulation of $I^4$ is sub-regular.
The total number of triangulations of $I^4$ adds up to 92\,487\,256.

On an AMD Ryzen 7 1700 \topcom takes 2345 seconds to enumerate the regular component, while \mptopcom requires 2598 seconds (single-threaded) and 457 seconds (with 10 processes), respectively.
Figure~\ref{fig:cube-timings-ryzen} shows how \mptopcom scales with the number of processes.
It shows that, on a standard desktop computer, already four processes, i.e., two workers, suffice for \mptopcom to be substantially faster than \topcom.
This should be compared with Figure~\ref{fig:cube-timings-cluster} which shows very similar behavior on our cluster.
The only exception is that, on that hardware, even a single-threaded version of \mptopcom beats \topcom.
The single-threaded \mptopcomone is the pure down-flip reverse search algorithm built without the overhead of \mts and \mpi.

Next we try to give an idea about which percentage of the total running time is spent on which subtask.
Table~\ref{table:callgrind:mptopcom} shows the values for \mptopcomone during the computation of all sub-regular triangulations of $I^4$.
For comparison the relative timings for \topcom are given in Table~\ref{table:callgrind:topcom}.
All these numbers were determined with \valgrind's tool \callgrind~\cite{valgrind}.
In both cases the bulk of the time is spent on finding and processing the flips.
For \mptopcom the major subtask is to determine the canonical representatives, while \topcom will explicitly compute the full orbits of each triangulation that it visits.
For both programs the cost for finding the initial triangulation (e.g., in Algorithm~\ref{alg:down-flip} before Step~\ref{step:reverse}) is negligible.
The picture changes entirely if one restricts the algorithms to enumerate regular triangulations only.
Then $\mptopcom$ spends $95\%$ of its time on solving linear programs, while that mark for $\topcom$ reaches $98\%$.
The remaining time is used in a similar fashion to that seen in
Tables~\ref{table:callgrind:mptopcom} and~\ref{table:callgrind:topcom}.
For larger examples the overall pattern stays the same, but the most costly subtasks tend to take up even higher proportions of the total running times.

\begin{table}[bh]\centering
\caption{Percentages of total running time spent by \mptopcomone for computing all (sub-regular) triangulations of the $4$-cube.  The value for flip processing further refined in second column.}
\label{table:callgrind:mptopcom}
\begin{tabular}{rrl}
\toprule
95\% &      & Process flips\\
& \multicolumn{1}{|r}{68\%} & Compute canonical representative\\
& \multicolumn{1}{|r}{20\%} & Partition into up- and down-flips\\
& \multicolumn{1}{|r}{7\%} & all remaining\\
3\%   &   & Check whether flip is edge of reverse search tree\\
2\%   &   & all remaining\\
\bottomrule
\end{tabular}
\end{table}

\begin{table}[bh]\centering
\caption{Percentages of total running time spent by \topcom for computing all triangulations of the $4$-cube.  The values for the most expensive subtasks are refined in the second and third columns.}
\label{table:callgrind:topcom}
\begin{tabular}{rrrl}
\toprule
99\% &  &     & Process flips\\
& \multicolumn{1}{|r}{92\%} & & Check whether class was already found\\
& \multicolumn{1}{|r}{} & \multicolumn{1}{|r}{82\%} & Enumerate orbits\\
& \multicolumn{1}{|r}{} & \multicolumn{1}{|r}{10\%} & all remaining\\
& \multicolumn{1}{|r}{7\%} & & all remaining\\
1\% & & & all remaining\\
\bottomrule
\end{tabular}
\end{table}

\subsection{Products of two simplices}
\label{subsec:product}
Another interesting class of point configurations are the products $\Delta_p\times\Delta_q$ of two simplices \cite[\S6.2]{Triangulations}.
The natural group action is by the product $\Sym(p+1)\times\Sym(q+1)$ of symmetric groups.
In tropical geometry, e.g., their regular subdivisions control the combinatorial types of tropical polytopes \cite[\S5.2]{Tropical+Book}.
The special case where one of the factors is one-dimensional, i.e., when the product of simplices is a prism, is fully understood \cite[\S6.2.1]{Triangulations}.
Therefore, in our experiments we restrict our attention to cases with $2\leq p\leq q$.
There is a formula for the number of all triangulations of $\Delta_2\times\Delta_q$ \cite[9.2.5]{Triangulations}, but this does not immediately yield the number of (semi-)regular triangulations or the number of orbits.

Figure \ref{fig:s3s7-timings-cluster} shows the speed for computing the triangulations in the
regular component of $\Delta_2 \times \Delta_6$ with \mptopcom on the cluster, depending on the number of processes.
This computation is medium size, i.e., a bit larger than the previous, and so it pays to use more processes.
The timings are as follows: one hour and 22 minutes with \mptopcom (10 processes), five hours with the single-threaded version \mptopcomone and eight days and 16 hours with \topcom.
A computation of this kind essentially marks the end of \topcom's range.

Our new results for $\Delta_2\times\Delta_6$ and $\Delta_3\times\Delta_4$ helped Schr\"oter \cite{Schroeter:1707.02814} to obtain new results on coarsest subdivisions of hypersimplices.
An attempt to handle $\Delta_3\times\Delta_5$ is currently under way (running on more than a hundred cores for some weeks).
So far it has found more than 900 million orbits of sub-regular triangulations.

\subsection{Dilated simplices}
\label{subsec:dilated-simplex}
A third class of point configurations is denoted as $k\cdot\Delta_d$.
These are the $\tbinom{n+d}{n}=\tbinom{n+d}{d}$ lattice points in the simplex $\Delta_d$ which is dilated by the factor~$k$.
For any polynomial in $d+1$ indeterminates which is homogeneous of degree $k$ the monomials correspond to points in the point configuration $k\cdot\Delta_d$.
In particular, the vertices of the Newton polytope form a subconfiguration.
It follows that the tropical hypersurfaces in the tropical $d$-torus $\TT^{d+1}/\RR\1$ of homogeneous degree $k$ are dual to regular subdivisions of $k\cdot\Delta_d$; see \cite[\S3.1]{Tropical+Book}.
The regular unimodular triangulations of $k\cdot\Delta_d$, which are necessarily full, correspond to those tropical hypersurfaces which are smooth.
For the first time, we computed the full triangulations of $3\cdot\Delta_3$, and these classify the smooth tropical cubics in $3$-space \cite[\S4.5]{Tropical+Book}.

\begin{theorem}\label{thm:cubic-surfaces}
  There are exactly $21\,125\,102$ orbits of regular and full triangulations of $3\cdot\Delta_3$ with respect to the natural action of the symmetric group of degree four. Out of these, $14\,373\,645$ are unimodular.
\end{theorem}

This is the largest experiment that we completed so far.
The computation took about four days on the Intel Xeon E5-2630 v2 cluster with 40 threads.

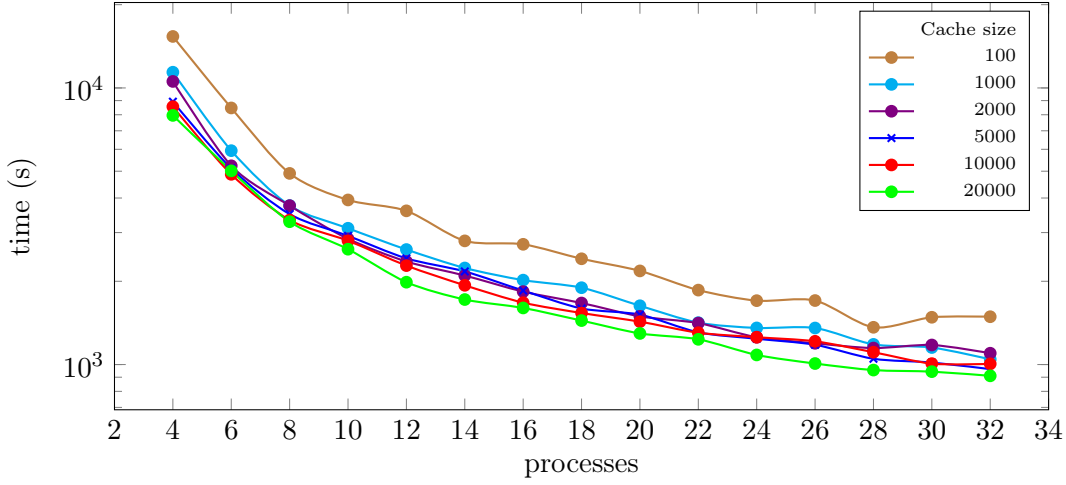
\begin{figure}[htb]
\begin{tikzpicture}
\begin{axis}[
      ylabel=\text{time (s)},
      xlabel=\text{processes},
      xmin = 2, xmax = 34,
      width = 14cm, height = 7cm,
      legend style={legend cell align=right,font=\tiny},
      ymode=log
   ]
   \addlegendimage{empty legend}
   \addlegendentry{Cache size}
   \addplot[smooth, mark=*, brown, thick] plot coordinates {
      (4, 15351.557)
      (6, 8466.334)
      (8, 4909.781)
      (10, 3936.779)
      (12, 3595.186)
      (14, 2799.886)
      (16, 2720.895)
      (18, 2412.165)
      (20, 2180.331)
      (22, 1858.7)
      (24, 1702.276)
      (26, 1703.843)
      (28, 1363.666)
      (30, 1482.363)
      (32, 1489.024)
   };
   \addlegendentry{100}
   \addplot[smooth, mark=*, cyan, thick] plot coordinates {
      (4, 11398.688)
      (6, 5932.758)
      (8, 3756.121)
      (10, 3109.799)
      (12, 2604.94)
      (14, 2234.16)
      (16, 2019.309)
      (18, 1896.14)
      (20, 1631.529)
      (22, 1417.296)
      (24, 1356.447)
      (26, 1355.483)
      (28, 1182.746)
      (30, 1152.325)
      (32, 1044.605)
   };
   \addlegendentry{1000}
   \addplot[smooth, mark=*, violet, thick] plot coordinates {
      (4, 10557.309)
      (6, 5229.791)
      (8, 3754.2)
      (10, 2848.44)
      (12, 2358.143)
      (14, 2097.936)
      (16, 1835.361)
      (18, 1669.056)
      (20, 1493.681)
      (22, 1412.278)
      (24, 1255.364)
      (26, 1193.696)
      (28, 1147.242)
      (30, 1176.816)
      (32, 1098.865)
   };
   \addlegendentry{2000}
   \addplot[smooth, mark=x, blue, thick] plot coordinates {
      (4, 8919.394)
      (6, 5113.214)
      (8, 3497.711)
      (10, 2919.88)
      (12, 2420.162)
      (14, 2170.678)
      (16, 1851.437)
      (18, 1594.905)
      (20, 1514.831)
      (22, 1307.553)
      (24, 1239.657)
      (26, 1178.837)
      (28, 1049.333)
      (30, 1016.882)
      (32, 962.591)
   };
   \addlegendentry{5000}
   \addplot[smooth, mark=*, red, thick] plot coordinates {
      (4, 8556.665)
      (6, 4871.901)
      (8, 3330.371)
      (10, 2803.193)
      (12, 2278.886)
      (14, 1934.356)
      (16, 1674.097)
      (18, 1535.849)
      (20, 1429.177)
      (22, 1303.431)
      (24, 1255.514)
      (26, 1213.473)
      (28, 1108.22)
      (30, 1007.654)
      (32, 1005.773)
   };
   \addlegendentry{10000}
   \addplot[smooth, mark=*, green, thick] plot coordinates {
      (4, 7956.224)
      (6, 5014.479)
      (8, 3283.934)
      (10, 2612.28)
      (12, 1985.174)
      (14, 1718.262)
      (16, 1601.925)
      (18, 1442.025)
      (20, 1295.639)
      (22, 1233.692)
      (24, 1083.393)
      (26, 1008.521)
      (28, 955.133)
      (30, 942.644)
      (32, 910.354)
   };
   \addlegendentry{20000}
\end{axis}
\end{tikzpicture}
\caption{Timings for enumerating the triangulations with \mptopcom of $\Delta_2\times\Delta_6$ in the regular component taken on the Intel Xeon E5-2630 v2 cluster with 64GB RAM per node, depending on the number of processes and cache sizes.  Each data point is the result of averaging over ten runs with the same parameters.}
\label{fig:s3s7-timings-cluster}
\end{figure}

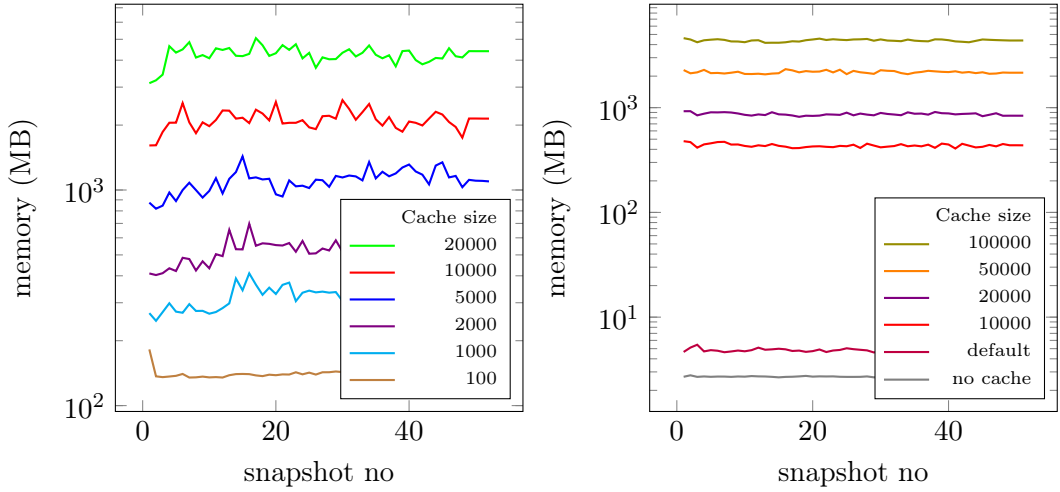
\begin{figure}[ht]
\begin{tikzpicture}
\begin{axis}[
      ylabel=\text{memory (MB)},
      xlabel=\text{snapshot no},
      width = 7cm, height = 7cm,
      legend style={legend cell align=right,font=\tiny},
      legend pos=south east,
      ymode=log
   ]
   \addlegendimage{empty legend}
   \addlegendentry{Cache size}
   \addplot[green ,thick] table {data/D6xD2/cache20000.out.plot};
   \addlegendentry{20000};
   \addplot[red ,thick] table {data/D6xD2/cache10000.out.plot};
   \addlegendentry{10000};
   \addplot[blue ,thick] table {data/D6xD2/cache5000.out.plot};
   \addlegendentry{5000};
   \addplot[violet ,thick] table {data/D6xD2/cache2000.out.plot};
   \addlegendentry{2000};
   \addplot[cyan ,thick] table {data/D6xD2/cache1000.out.plot};
   \addlegendentry{1000};
   \addplot[brown ,thick] table {data/D6xD2/cache100.out.plot};
   \addlegendentry{100};
\end{axis}
\end{tikzpicture}
\begin{tikzpicture}
\begin{axis}[
      ylabel=\text{memory (MB)},
      xlabel=\text{snapshot no},
      width = 7cm, height = 7cm,
      legend style={legend cell align=right,font=\tiny},
      legend pos=south east,
      ymode=log,
   ]
   \addlegendimage{empty legend}
   \addlegendentry{Cache size}
   \addplot[olive ,thick] table {data/3D3/cache100000.out.plot};
   \addlegendentry{100000};
   \addplot[orange ,thick] table {data/3D3/cache50000.out.plot};
   \addlegendentry{50000};
   \addplot[violet ,thick] table {data/3D3/cache20000.out.plot};
   \addlegendentry{20000};
   \addplot[red ,thick] table {data/3D3/cache10000.out.plot};
   \addlegendentry{10000};
   \addplot[purple ,thick] table {data/3D3/defaultCache.out.plot};
   \addlegendentry{default};
   \addplot[gray ,thick] table {data/3D3/noCache.out.plot};
   \addlegendentry{no cache};
\end{axis}
\end{tikzpicture}
\caption{Memory usage determined with \valgrind.  We took 52 \massif snapshots for $\Delta_2\times\Delta_6$ (left) and $3\cdot\Delta_3$ (right).}\label{fig:massif-cache}
\end{figure}

\subsection{Using more memory for caching}
On the one hand, as a key benefit, the memory consumption of the down-flip reverse search algorithm allows for an excellent a priori estimate which is also quite low; cf.\ Theorem~\ref{thm:complexity}.
On the other hand this approach results in a considerable amount of duplication.
To avoid at least some of this it is natural to employ caching, as explained in Section~\ref{sec:implementation}.
Here we want to report on some experiments concerning the impact of caching on the overall running time.

First we investigate the medium-size example $\Delta_2\times\Delta_6$ from Section~\ref{subsec:product}.
It has $533\, 242$ sub-regular triangulations up to symmetry.
Figure~\ref{fig:s3s7-timings-cluster} shows how the running-time depends on the
number of workers and on the cache sizes of the three main caches introduced in
Section~\ref{sec:implementation}. Their size can be varied individually, but
for this plot we gave all of them the same size.
As a default \mptopcom stores 2000 triangulations in each of these caches.
In this case increasing the cache size from 100 to 2000 reduces the running time by about one third, independent of the number of threads.
Increasing the cache to $20\, 000$ only results in a further reduction by another $5\%$.

Figure~\ref{fig:massif-cache} shows how the total amount of memory consumed depends on the cache sizes.
The measurements have been taken by \valgrind's tool \massif which records memory snapshots in fixed time intervals \cite{valgrind}.
The left hand side corresponds to the computation in Figure~\ref{fig:s3s7-timings-cluster} for $\Delta_2\times\Delta_6$.
By and large the overall memory consumption depends linearly on the cache sizes.
Note that the sizes of the triangulations as well as the sizes of the various objects for the cache values vary.
So some fluctuations should be expected.
Indeed, this is visible for $\Delta_2\times\Delta_6$, which is not very large.
In the much larger example $3\cdot\Delta_3$ from Section~\ref{subsec:dilated-simplex} we see a similar behavior, but the fluctuations are nearly gone.

\begin{table}[th]\centering
\caption{Summary of enumerations.}
\label{tab:summary}
\begin{tabular}{crrrrrr}
\toprule
Points $P$ & $n=\#P$ & $d=\dim{P}$ & $n-d$ & $\#\group$ & \multicolumn{2}{c}{\#triangulation orbits} \\
           &         &             &       &            & (full) regular & sub-regular        \\
\midrule
$I^3$                    &  8 & 3 &  5 &      48 &          6 &         6 \\
$I^4$                    & 16 & 4 & 12 &     384 &   235\,277 &  247\,451 \\
%
$I^5$                    & 32 & 5 & 27 &   3\,840 &          &           \\[.2cm]
$\Delta_2\times\Delta_2$ &  9 & 4 &  5 &       36 &        5 &         5 \\
$\Delta_2\times\Delta_3$ & 12 & 5 &  7 &      144 &       35 &        35 \\
$\Delta_2\times\Delta_4$ & 15 & 6 &  9 &      720 &      530 &       530 \\
$\Delta_2\times\Delta_5$ & 18 & 7 & 11 &   4\,320 &  13\,621 &   13\,629 \\
$\Delta_2\times\Delta_6$ & 21 & 8 & 13 &  30\,240 & 531\,862 &  533\,242 \\[.2cm]
$\Delta_3\times\Delta_3$ & 16 & 6 & 10 &      576 &   7\,869 &    7\,955 \\
%
%
$\Delta_3\times\Delta_4$ & 20 & 7 & 13 &   2\,880 & 7\,051\,957 & 7\,402\,421 \\
$\Delta_3\times\Delta_5$ & 24 & 8 & 16 &  17\,280 &   & $> 9 \cdot 10^8$\\[0.2cm]

$2\Delta_3$              & 10 & 3 &  7 &       24 &           15 & 59\\
$3\Delta_3$              & 20 & 3 & 17 &       24 & 21\,125\,102 & 925\,148\,763\\
$4\Delta_3$              & 35 & 3 & 32 &       24 &              & \\
\bottomrule
\end{tabular}
\end{table}

\section{Concluding remarks}\label{sec:concluding_remarks}
\noindent
Table~\ref{tab:summary} also contains some empty rows, where we do not know the respective number of triangulations.
Most of these will be out of reach for the current implementations, including \mptopcom.
The reason for listing these nonetheless is to give a feel for the orders of magnitude involved.
One main complexity parameter for enumerating triangulations is the difference $n-d$ of the number of points and the dimension.
This is also one plus the dimension of the secondary fan, modulo linealities.
Our experiments suggest that, as a very rough estimate, the range for \topcom seems to be limited by $n-d \approx 13$.
This bar is raised substantially by \mptopcom to cover point configurations with $n-d = 17$ such as $3\cdot\Delta_3$.
It is an interesting question if \mptopcom can, e.g., deal with $I^3\times\Delta_2$ where $n-d=19$.
That particular point configuration played a role in work of Orden and Santos \cite{OrdenSantos:2003} on efficient triangulations of cubes; see also \cite[\S6.3.3]{Triangulations}.

The empty rows of Table~\ref{tab:summary} show some cases which seem to be rather difficult challenges, with the current techniques.
This includes the five-dimensional cube $I^5$ or the dilated simplex $4\cdot\Delta_3$.
Proving results about their triangulations might require clever strategies for random probing.

Another direction which looks promising is to investigate the triangulations of the cyclic polytopes.
This is related to the higher Stasheff--Tamari orders which were introduced by Kaparanov and Voevodsky \cite{KapranovVoevodsky:Stasheff-Tamari} and studied, e.g., by Edelman and Reiner \cite{EdelmanReiner:Stasheff-Tamari}; see Rambau and Reiner \cite{RambauReiner:Stasheff-Tamari} for a survey.

\bibliographystyle{amsplain}
\bibliography{mptopcom}

\providecommand{\bysame}{\leavevmode\hbox to3em{\hrulefill}\thinspace}
\providecommand{\MR}{\relax\ifhmode\unskip\space\fi MR }
\providecommand{\MRhref}[2]{%
  \href{http://www.ams.org/mathscinet-getitem?mr=#1}{#2}
}
\providecommand{\href}[2]{#2}
\begin{thebibliography}{10}

\bibitem{AD17}
David Avis and Luc Devroye, \emph{An analysis of budgeted parallel search on
  conditional {Galton-Watson} trees}, 2017, Preprint \arXiv{1703.10731}.

\bibitem{AF93}
David {Avis} and Komei {Fukuda}, \emph{{Reverse search for enumeration.}},
  {Discrete Appl. Math.} \textbf{65} (1996), no.~1-3, 21--46 (English).

\bibitem{mts_tutorial}
David Avis and Charles Jordan, \emph{A parallel framework for reverse search
  using \mts}, 2016, Preprint \arXiv{1610.07735}.

\bibitem{mplrs}
\bysame, \emph{\mplrs: A scalable parallel vertex/facet enumeration code},
  Mathematical Programming Computation (2018), to appear.

\bibitem{BremnerDutourSchuermann:2009}
David Bremner, Mathieu Dutour~Sikiri\'c, and Achill Sch\"urmann,
  \emph{Polyhedral representation conversion up to symmetries}, Polyhedral
  computation, CRM Proc. Lecture Notes, vol.~48, Amer. Math. Soc., Providence,
  RI, 2009, pp.~45--71. \MR{2503772}

\bibitem{BMFN99}
Adrian Br\"{u}ngger, Ambros Marzetta, Komei Fukuda, and Jurg Nievergelt,
  \emph{The parallel search bench {ZRAM} and its applications}, Annals of
  Operations Research \textbf{90} (1999), 45--63.

\bibitem{Triangulations}
Jes{\'u}s~A. De~Loera, J{\"o}rg Rambau, and Francisco Santos,
  \emph{Triangulations}, Algorithms and Computation in Mathematics, vol.~25,
  Springer-Verlag, Berlin, 2010, Structures for algorithms and applications.
  \MR{2743368 (2011j:52037)}

\bibitem{EdelmanReiner:Stasheff-Tamari}
Paul~H. {Edelman} and Victor {Reiner}, \emph{{The higher Stasheff-Tamari
  posets.}}, {Mathematika} \textbf{43} (1996), no.~1, 127--154 (English).

\bibitem{FKL05}
J.-A. Ferrez, K.~Fukuda, and Th.~M. Liebling, \emph{Solving the fixed rank
  convex quadratic maximization in binary variables by a parallel zonotope
  construction algorithm}, European J. Oper. Res. \textbf{166} (2005), no.~1,
  35--50. \MR{2128976}

\bibitem{cddlib}
Komei Fukuda, \emph{\cddlib, version 0.94h},
  \url{http://www.inf.ethz.ch/personal/fukudak/cdd_home/}, 2015.

\bibitem{polymake}
Ewgenij Gawrilow and Michael Joswig, \emph{\polymake: a framework for analyzing
  convex polytopes}, Polytopes---combinatorics and computation (Oberwolfach,
  1997), DMV Sem., vol.~29, Birk\-h\"au\-ser, Basel, 2000, pp.~43--73.
  \MR{MR1785292 (2001f:52033)}

\bibitem{GKZ}
I.~M. Gel{\cprime}fand, M.~M. Kapranov, and A.~V. Zelevinsky,
  \emph{Discriminants, resultants and multidimensional determinants}, Modern
  Birkh\"auser Classics, Birkh\"auser Boston Inc., Boston, MA, 2008, Reprint of
  the 1994 edition. \MR{MR2394437 (2009a:14065)}

\bibitem{Imai:2002}
Hiroshi Imai, Tomonari Masada, Fumihiko Takeuchi, and Keiko Imai,
  \emph{Enumerating triangulations in general dimensions}, Internat. J. Comput.
  Geom. Appl. \textbf{12} (2002), no.~6, 455--480. \MR{1945594}

\bibitem{Jensen:2016}
Anders~N. Jensen, \emph{An implementation of exact mixed volume computation},
  Mathematical Software -- ICMS 2016: 5th International Conference, Berlin,
  Germany, July 11-14, 2016, Proceedings (Gert-Martin Greuel, Thorsten Koch,
  Peter Paule, and Andrew Sommese, eds.), Springer International Publishing,
  Cham, 2016, pp.~198--205.

\bibitem{Jensen:1601.02818}
\bysame, \emph{Tropical homotopy continuation}, 2016, Preprint
  \arXiv{1601.02818}.

\bibitem{gfan}
\bysame, \emph{\gfan, a software system for {G}r{\"o}bner fans and tropical
  varieties, version 0.6}, Available at
  \url{http://home.imf.au.dk/jensen/software/gfan/gfan.html}, 2017.

\bibitem{KapranovVoevodsky:Stasheff-Tamari}
Mikhail~M. Kapranov and Vladimir~A. Voevodsky, \emph{Combinatorial-geometric
  aspects of polycategory theory: pasting schemes and higher {B}ruhat orders
  (list of results)}, Cahiers Topologie G\'eom. Diff\'erentielle Cat\'eg.
  \textbf{32} (1991), no.~1, 11--27, International Category Theory Meeting
  (Bangor, 1989 and Cambridge, 1990). \MR{1130400}

\bibitem{Tropical+Book}
Diane Maclagan and Bernd Sturmfels, \emph{Introduction to tropical geometry},
  Graduate Studies in Mathematics, vol. 161, American Mathematical Society,
  Providence, RI, 2015. \MR{3287221}

\bibitem{valgrind}
Nicholas Nethercote and Julian Seward, \emph{Valgrind: A framework for
  heavyweight dynamic binary instrumentation}, Proceedings of the 28th ACM
  SIGPLAN Conference on Programming Language Design and Implementation (New
  York, NY, USA), PLDI '07, ACM, 2007, pp.~89--100.

\bibitem{OrdenSantos:2003}
David Orden and Francisco Santos, \emph{Asymptotically efficient triangulations
  of the {$d$}-cube}, Discrete Comput. Geom. \textbf{30} (2003), no.~4,
  509--528. \MR{2013970}

\bibitem{Pfeifle:MOAE}
Julian Pfeifle, \emph{Secondary polytope of the ``mother of all examples''},
  Electronic Geometry Models (2000), No.~2000.09.033.

\bibitem{PfeifleRambau03}
Julian Pfeifle and J{\"o}rg Rambau, \emph{Computing triangulations using
  oriented matroids}, Algebra, geometry, and software systems, Springer,
  Berlin, 2003, pp.~49--75. \MR{2011753 (2004i:68233)}

\bibitem{Pournin:2013}
Lionel Pournin, \emph{The flip-graph of the 4-dimensional cube is connected},
  Discrete Comput. Geom. \textbf{49} (2013), no.~3, 511--530. \MR{3038527}

\bibitem{PourninLiebling:2007}
Lionel Pournin and Thomas~M. Liebling, \emph{Constrained paths in the
  flip-graph of regular triangulations}, Comput. Geom. \textbf{37} (2007),
  no.~2, 134--140. \MR{2310598}

\bibitem{TOPCOM}
J\"org Rambau, \emph{\topcom: triangulations of point configurations and
  oriented matroids}, Mathematical software ({B}eijing, 2002), World Sci.
  Publ., River Edge, NJ, 2002, pp.~330--340. \MR{1932619}

\bibitem{RambauReiner:Stasheff-Tamari}
J\"org {Rambau} and Victor {Reiner}, \emph{{A survey of the higher
  Stasheff-Tamari orders.}}, {Associahedra, Tamari lattices and related
  structures. Tamari memorial Festschrift}, Basel: Birkh\"auser, 2012,
  pp.~351--390 (English).

\bibitem{permlib}
Thomas {Rehn} and Achill {Sch{\"u}rmann}, \emph{{C++ tools for exploiting
  polyhedral symmetries.}}, {Mathematical software -- ICMS 2010. Third
  international congress on mathematical software, Kobe, Japan, September
  13--17, 2010. Proceedings}, Berlin: Springer, 2010, pp.~295--298 (English).

\bibitem{Schroeter:1707.02814}
Benjamin Schr\"oter, \emph{Multi-splits and tropical linear spaces from nested
  matroids}, 2017, Preprint \arXiv{1707.02814}.

\bibitem{Seress:2003}
\'Akos Seress, \emph{Permutation group algorithms}, Cambridge Tracts in
  Mathematics, vol. 152, Cambridge University Press, Cambridge, 2003.
  \MR{1970241}

\bibitem{mpi}
Marc Snir, Steve Otto, Steven Huss-Lederman, David Walker, and Jack Dongarra,
  \emph{\mpi -- the complete reference, vol 1: The \mpi core}, 2nd ed., MIT
  Press, 1998.

\bibitem{Weibel10}
Christophe Weibel, \emph{Implementation and parallelization of a reverse-search
  algorithm for {Minkowski} sums}, 2010 Proceedings of the Twelfth Workshop on
  Algorithm Engineering and Experiments (ALENEX), 2010, pp.~34--42.

\end{thebibliography}

\end{document}